\documentclass[a4paper,12pt]{article}

\makeatletter
 
  \@addtoreset{equation}{section}
 \makeatother
\usepackage{amsfonts}
\usepackage{amssymb}
\usepackage{mathrsfs}
\usepackage{amsmath,amsthm}
\usepackage{amsmath}

\usepackage{graphicx}
\usepackage{color}
\usepackage{indentfirst}

\begin{document}

\newtheorem{thm}{Theorem}
\newtheorem{lemma}{Lemma}
\newtheorem{prop}{Proposition}
\newtheorem{defn}{Definition}
\newtheorem{rem}{Remark}
\newtheorem{step}{Step}
\newtheorem{cor}{Corollary}

\newcommand{\Cov}{\mathop {\rm Cov}}
\newcommand{\Var}{\mathop {\rm Var}}
\newcommand{\E}{\mathop {\rm E}}
\newcommand{\const }{\mathop {\rm const }}
\everymath {\displaystyle}

\newcommand{\ruby}[2]{
\leavevmode
\setbox0=\hbox{#1}
\setbox1=\hbox{\tiny #2}
\ifdim\wd0>\wd1 \dimen0=\wd0 \else \dimen0=\wd1 \fi
\hbox{
\kanjiskip=0pt plus 2fil
\xkanjiskip=0pt plus 2fil
\vbox{
\hbox to \dimen0{
\small \hfil#2\hfil}
\nointerlineskip
\hbox to \dimen0{\mathstrut\hfil#1\hfil}}}}

\def\qedsymbol{$\blacksquare$}
\renewcommand{\thefootnote }{\fnsymbol{footnote}}
\renewcommand{\refname }{References}

\everymath {\displaystyle}

\title{
Stock Price Fluctuations in an Agent-Based Model with Market Liquidity
}
\author{Takashi Kato
\footnote{Division of Mathematical Science for Social Systems, 
              Graduate School of Engineering Science, 
              Osaka University, 
              1-3, Machikaneyama-cho, Toyonaka, Osaka 560-8531, Japan, 
E-mail: \texttt{kato@sigmath.es.osaka-u.ac.jp}}
}

\date{}
\maketitle 

\begin{abstract}
We study an agent-based stock market model with heterogeneous agents and friction. 
Our model is based on that of \cite {Foellmer-Schweizer}: 
the process of a stock price in a discrete-time framework is determined by temporary equilibria 
via agents' excess demand functions, 
and the diffusion approximation approach is applied to characterize 
the continuous-time limit (as transaction intervals shorten) as 
a solution of the corresponding stochastic differential equation (SDE). 
In this paper we further make the assumption that 
some of the agents are bound by either short sale constraints or budget constraints. 
Then we show that the continuous-time process of the stock price can be derived from 
a certain SDE with oblique reflection. 
Moreover we find that the short sale (respectively, budget) constraint causes 
overpricing (respectively, underpricing). \\
\footnote[0]{Mathematical Subject Classification (2010) \  91B24, 91B69, 60F17}\\
\footnote[0]{JEL Classification (2010) \ D53, C69}
{\bf Keywords}: 
Agent-based models, Liquidity problems, Short sale/budget constraints, 
Stochastic differential equations with oblique reflection, The Skorokhod problem. 
\end{abstract}

\section{Introduction}\label{sec_intro}

It is usual in mathematical finance to describe the price evolution of a risky asset such as a stock 
by a diffusion process. 
Geometric Brownian motion (GBM) is one of the most standard such processes for price fluctuation. 
Because of its simplicity and convenience, the GBM model is widely used in the context of 
option pricing/hedging, optimal investment, and many other financial problems. 
An important theme is to justify GBM from the economic viewpoint. 
For instance, a heuristic equilibrium argument for GBM is discussed in \cite {Samuelson}. 
The justification of GBM as the rational expectations equilibrium is discussed in \cite{Bick} and \cite{Kreps}. 

Recently there have been various studies of agent-based market models to explain 
the fluctuation of a price process. 
One representative study is the microeconomic approach of 
\cite {Foellmer} and \cite {Foellmer-Schweizer}: 
The process of the stock price is 
first given as a sequence of temporary price equilibria in a discrete-time market model 
with heterogeneous agents and then 
the price process in a continuous-time model is 
derived as the limit as the transaction time intervals shorten. 
Let us introduce the outline of the model of \cite {Foellmer-Schweizer}. 
Let $I$ be the set of agents in the market and 
$\hat{e}^{n, i}_k(p, \omega )$ be trader $i$'s excess demand function for a proposed stock price $p$ at the time $k/n$, 
where $n\in \Bbb {N}$ and trades are executed at times $t = 0, 1/n, 2/n, \ldots $. 
The parameter $\omega $ is a sample point in the underlying probability space $(\Omega , \mathcal {F}, P)$. 
The stock price process $(S^n_{k/n})_{k=0, 1, 2, \ldots }$ is given as follows: 
At the initial time $t = 0$, the stock price is given by $S^n_0 = s_0$. 
Then agents exhibit their excess demand $\hat{e}^n_{i, 0}(\cdot , \omega )$ and 
the next price is determined as a temporary equilibrium, that is, the solution $p^*$ of 
\begin{eqnarray*}
\sum _{i\in I}\hat{e}^n_{i, 0}(p^*(\omega ), \omega ) = 0. 
\end{eqnarray*}
After transactions at $t = 0$, the stock price changes to $S^n_{1/n} = p^*$. 
Similarly, during the trading period $k/n$, 
the stock price before transactions is given by $S^n_{k/n}$ and 
agents' excess demands $(\hat{e}^{n, i}_k(\cdot , \omega ))_{i\in I}$ make 
the stock price change to $S^n_{(k+1)/n}$. 
Finally, the process $(S^n_{k/n})_k$ is given as the solution of
\begin{eqnarray}\label{eq_system_S}
\sum _{i\in I}\hat{e}^{n, i}_k(S^n_{(k+1)/n}(\omega ), \omega ) = 0, \ \ k = 0, 1, 2, \ldots , \ \ S^n_0 = s_0. 
\end{eqnarray}
For mathematical convenience we rewrite (\ref {eq_system_S}) by using the log price 
$X^n_k = \log S^n_k$ of the stock, obtaining 
\begin{eqnarray}\label{eq_system_X}
\sum _{i\in I}e^{n, i}_k(X^n_{(k+1)/n}(\omega ), \omega ) = 0, \ \ k = 1, 2, \ldots , \ \ X^n_0 = x_0, 
\end{eqnarray}
where $e^{n, i}_k(x, \omega ) = \hat{e}^{n, i}_k(e^x, \omega )$ and $x_0 = \log s_0$. 
In \cite {Foellmer} and \cite {Foellmer-Schweizer}, 
the individual excess demand function $e^{n, i}_k$ is assumed to be given by
\begin{eqnarray}\label{def_e_f}
\begin{split}
e^{n, i}_k(x, \omega ) &= \alpha ^{n, i}_k(f^{n, i}_k(x, \omega ) - x) + \delta ^{n, i}_k(\omega ), \\
f^{n, i}_k(x, \omega ) &= X^n_{k/n} + \beta ^{n, i}_k(X^n_{k/n} - F_i) + \gamma ^{n, i}_k(X^n_{k/n} - x). 
\end{split}
\end{eqnarray}
The parameter $\delta ^{n, i}_k(\omega )$ denotes the liquidity demand, 
$f^{n, i}_k(x)$ is the reference level of agent $i$, and 
$F_i$ is agent $i$'s individual perception of the fundamental (log-)value. 
For a more precise economic interpretation of (\ref {def_e_f}), 
see Example 3) in Section 2 of \cite {Foellmer-Schweizer}. 
In this case, (\ref {eq_system_X}) can be rewritten using the stochastic difference equation 
\begin{eqnarray}\label{diff_eq_X}
X^n_{(k+1)/n} - X^n_{k/n} = \bar{\beta }^n_k(X^n_{k/n} - \bar{F}) + \bar{\delta }^n_k, \ \ 
k = 0, 1, 2, \ldots , 
\end{eqnarray}
where 
\begin{eqnarray*}
\bar{\beta }^n_k = \frac{\sum _{i\in I}\alpha ^{n, i}_k\beta ^{n, i}_k}{\bar{\alpha }^n_k}, \ \ 
\bar{F}^n_k = \frac{\sum _{i\in I}\alpha ^{n, i}_k\beta ^{n, i}_kF_i}{\bar{\alpha }^n_k}, \ \ 
\bar{\delta }^n_k = \frac{\sum _{i\in I}\delta ^{n, i}_k}{\bar{\alpha }^n_k}, \ \ 
\bar{\alpha }^n_k = \sum _{i\in I}\alpha ^{n, i}_k(1 + \gamma ^{n, i}_k). 
\end{eqnarray*}
Let $(X^n_t)_{t\geq 0}$ be an interpolated process of $(X^n_{k/n})_k$ such as either 
\begin{eqnarray}\label{step_interpolation}
X^n_t = X^n_{k/n}\ \ t\in (k/n, (k+1)/n) 
\end{eqnarray}
or
\begin{eqnarray}\label{linear_interpolation}
X^n_t = (nt - k)X^n_{(k+1)/n} + (k + 1 - nt)X^n_{k/n}\ \ t\in (k/n, (k+1)/n). 
\end{eqnarray}
Then, under some mathematical assumptions, 
the process $(X^n_t)_t$ converges to an Ornstein--Uhlenbeck (OU) process of the form 
\begin{eqnarray*}
dX_t = \bar{\beta }(X_t - \bar{F})dt + \bar{\sigma }dB_t, \ \ X_0 = x_0. 
\end{eqnarray*}
(This is a simplified version of the result of \cite {Foellmer-Schweizer}: 
They also treat a generalized process in a random environment.) 
This implies that the continuous-time stock price $S_t = \exp (X_t)$ is 
a geometric OU process. 

It is meaningful to consider the above diffusion approximation 
approach to derive the continuous-time process in a more general framework. 
A diffusion approximation for solutions of stochastic difference equation in the following form 
\begin{eqnarray}\label{diff_eq_X_general}
X^n_{(k+1)/n} - X^n_{k/n} = \frac{1}{\sqrt{n}}F^n_k(X^n_{k/n}, \omega ) + 
\frac{1}{n}G^n_k(X^n_{k/n}, \omega ), 
\end{eqnarray}
with $\E [F^n_k(x)] = 0$, is studied in \cite {Kushner-Huang} and \cite {Watanabe1}--\cite {Watanabe3} 
under some mixing conditions. 
In \cite {Kato}, a case of the functional difference equation 
\begin{eqnarray}\label{diff_eq_X_general2}
X^n_{(k+1)/n} - X^n_{k/n} = \frac{1}{\sqrt{n}}F^n_k((X^n_r)_{r\leq k/n}, \omega ) + 
\frac{1}{n}G^n_k((X^n_r)_{r\leq k/n}, \omega ) 
\end{eqnarray}
is studied under strong mixing conditions and a certain additional dimensional condition. 
By using these results, we can apply the diffusion approximation approach 
to the agent-based market model for more general excess demand functions 
and derive several stock price models based on the framework of \cite {Foellmer-Schweizer}. 

The aim of this paper is to construct an agent-based model of stock prices with market liquidity problems. 
In the real market, although there are agents who can buy and sell the stock freely to some extent, 
there also exist agents who cannot trade to their own satisfaction because of a shortage of cash, being prohibited from short selling, and so on. 
To consider how such a liquidity problem affects things, 
we construct a market model based on \cite {Foellmer-Schweizer} under the following constraints: 
\begin{description}
 \item[(I)] \ Some of the agents cannot sell the more of the stock than the number of shares held (a short sale constraint), 
 \item[(II)] \ Some of the agents cannot buy more of the stock higher than allowed by their budget. (a budget constraint). 
\end{description}
In each case (I)--(II), we will show that 
the continuous-time process of the stock price 
which is derived by shortening the transaction intervals 
is the solution of a certain stochastic differential equation with oblique reflection (SDER). 
Moreover, the value of the stock price under the short sale constraint 
is larger than it would be without such a constraint. 
The effect of the short sale constraint is discussed in 
\cite {Bai-Chang-Wang}, \cite {Diamond-Verrecchia}, \cite {Harrison-Kreps} 
and the references therein, 
and our result is consistent with a common expectation, viz.,  that 
the short sale constraint causes overpricing 
(nonetheless, \cite {Bai-Chang-Wang} pointed out that 
whether short sale constraints will always lead to overpricing is far from certain). 
On the other hand, we will also show that a budget constraint drives the stock price down. 

We now fix some notation. 
For an interval $A\subset [0, \infty )$, 
we denote by $C(A ; \Bbb {R}^d)$ the set of continuous functions from $A$ to $\Bbb {R}^d$ 
and we use the abbreviations $\mathcal {C}^d_T = C([0, T] ; \Bbb {R}^d)$ and 
$\mathcal {C}^d = C([0, \infty ) ; \Bbb {R}^d)$ 
(when $d = 1$, we simply write $\mathcal {C}_T$ and $\mathcal {C}$.) For $w\in \mathcal {C}^d$, 
we set $|w|_T = \sup _{0\leq t\leq T}|w(t)|$ and 
$|w|_\infty = \sup _{t\geq 0}|w(t)|$, where $|\cdot |$ is the Euclidean norm. 
We also define the following subspaces of $\mathcal {C}^d_T$. 
\begin{eqnarray*}
\mathcal {C}^d_{T,+} &=& \{ w = (w^i)^d_{i=1}\in \mathcal {C}^d_T\ ; \ w^i(t)\geq 0, \ \ i = 1, \ldots , d\}, \\
\mathcal {C}^d_{T,\uparrow 0} &=& \{ w = (w^i)^d_{i=1}\in \mathcal {C}^d_T\ ; \ w^i(0) = 0, w^i(t) 
\mbox { is non-decreasing in }t, \ \ i = 1 , \ldots , d \}, 
\end{eqnarray*}
and similarly $\mathcal {C}^d_+$ and $\mathcal {C}^d_{\uparrow 0}$ for $T = \infty $. 
We introduce the canonical $\sigma $-algebra $\mathcal {B}_t = \sigma ( w(s)\ ; \ s\leq t )$ 
of $\mathcal {C}$. 
We often consider the space of $\Bbb {R}^{1+N}$-valued functions and then 
we start the index at zero, i.e., $w(t) = (w^i(t))^N_{i = 0}$. 
The notation $(x)_+$ stands for the positive part of $x$, that is, 
$(x)_+ = \max \{ x, 0 \}$. 

\section{Main Results}\label{sec_model}
\subsection{Model I: Stock Price Model with Short Sale Constraints}\label{subsec_short_sale}
Let $I = \{1, \ldots , N\}$ be a finite set of agents who are active in the market 
which consists of a single stock. 
We assume that each agent $i\in I$ always stays in the market and 
no new agents enter. 
First we consider the discrete trading case, with 
market clearing times $t = 1/n, 2/n, \ldots $ for $n \in \Bbb {N} = \{1, 2, 3, \ldots \}$. 
We denote by $X^n = (X^n_t)_{t\geq 0}$ the log-price process of a stock, and 
by $\varphi ^{n, i} = (\varphi ^{n, i}_t)_{t\geq 0}$ the number of shares of the stock held by agent $i\in I$. 
The fluctuations of the processes $X^n$ and $\varphi ^n = (\varphi ^{n, i})^N_{i = 1}$ are found as follows.
At the initial time $t = 0$, every agent $i\in I$ has $\varphi ^{n, i}_0 = \Phi ^i\geq 0$ shares of a stock 
and the initial log-price of the stock is $X^n_0 = x_0\in \Bbb {R}$. 
After trading is finished for the period $t = k/n$, 
the log-price $X^n_t$, $t\in (k/n, (k+1)/n]$ is determined by the market clearing condition 
\begin{eqnarray}\label{market_clearing}
\sum _{i\in I}e^{n, i}_k(X^n_{(k+1)/n}, X^n, \omega ; \varphi ^{n, i}_{k/n}) = 0 
\end{eqnarray}
and linear interpolation (\ref {linear_interpolation}), 
where $e^{n, i}_k(x, w, \omega  ; \varphi ) : 
\Bbb {R}\times \mathcal {C}\times \Omega ^n \times \Bbb {R} \longrightarrow \Bbb {R}$ 
is the excess demand function of $i\in I$ at $t = k/n$ 
on the underlying probability space $(\Omega ^n, \mathcal {F}^n, P^n)$, 
and $\varphi ^{n, i}_k$ is the quantity of stock held by $i$. 
Here $e^{n, i}_k(x, w, \omega  ; \varphi )$ is assumed to be 
$\mathcal {B}(\Bbb {R})\otimes \mathcal {B}_{k/n}\otimes \mathcal {F}^n\otimes \Bbb {R}$-measurable, so that 
\begin{eqnarray*}
e^{n, i}_k(x, w, \omega  ; \varphi ) = e^{n, i}_k(x, w(\cdot \wedge (k/n)), \omega  ; \varphi ). 
\end{eqnarray*}
After determining the log-price up to $t = (k+1)/n$ by (\ref {market_clearing}), 
the agents' holdings of the stock are set by 
\begin{eqnarray}\label{fluctuation_phi}
\varphi ^{n, i}_{(k+1)/n} = \varphi ^{n, i}_{k/n} + e^{n, i}_k(X^n_{(k+1)/n}, X^n, \omega ; \varphi ^{n, i}_{k/n}) 
\end{eqnarray}
and 
\begin{eqnarray}\label{linear_interpolation_phi}
\varphi ^{n, i}_t = (nt - k)\varphi ^{n, i}_{(k+1)/n} + (k + 1 - nt)\varphi ^{n, i}_{k/n},\ \ t\in (k/n, (k+1)/n) 
\end{eqnarray}
for $i = 1, \ldots , N$. 
We remark that (\ref {market_clearing}) implies that the total number of shares of the stock in the market is constant, i.e.,
$\sum ^N_{i = 1}\varphi ^{n, i}_t = \sum ^N_{i = 1}\Phi ^i$ for any $n$ and $t$. 

All agents have their individual excess demand $\tilde{e}^{n, i}_k( x, w, \omega )$ 
before considering friction, but 
they do not always exhibit $\tilde{e}^{n, i}_k( x, w, \omega )$ itself to the market. 
Agents are divided into two groups, $I = I_1\cup I_2$, where 
$I_1 = \{1, \ldots , N_1\}$ and $I_2 = \{ N_1 + 1, \ldots , N \}$ for some $1\leq N_1\leq N$. 
Agents in the first group are prohibited from 
selling short,
thus their excess demand at $t = k/n$ is not lower than $\varphi ^{n, i}_{k/n}$ 
and the process of their stock holdings is always non-negative. 
Agents in the second group $I_2$ are allowed to sell short, 
so they can exhibit the value $\tilde{e}^{n, i}_k(\cdot , X^n, \omega )$ itself as their excess demand. 
Then the excess demand function is finally defined as 
\begin{eqnarray}\label{def_e}
e^{n, i}_k(x, w, \omega  ; \varphi ) = 
\left\{
\begin{array}{ll}
 \max \{ \tilde {e}^{n, i}_k(x, w, \omega ), -\varphi \} , 	& i\in I_1	\\
 \tilde {e}^{n, i}_k(x, w, \omega ), 	& i\in I_2. 
\end{array}
\right.
\end{eqnarray}

We divide $\tilde{e}^{n, i}_k(x, w)$ into two parts. 
\begin{eqnarray}\label{form_e}
\tilde{e}^{n, i}_k(x, w) = 
f^{n, i}_k(x, w) + g^{n, i}_k(w). 
\end{eqnarray}
Here, 
\begin{eqnarray*}
f^{n, i}_k(x, w) = \tilde{e}^{n, i}_k(x, w) - \tilde{e}^{n, i}_k(w(k/n), w), \ \ 
g^{n, i}_k(w) = \tilde{e}^{n, i}_k(w(k/n), w). 
\end{eqnarray*}
The function $g^{n, i}_k(w)$ is the excess demand when the price of the stock does not change. 
We further divide $g^{n, i}_k(w)$ into 
\begin{eqnarray*}
g^{n, i}_k(w) = \frac{1}{n}\bar{g}^{n, i}_k(w) + \frac{1}{\sqrt{n}}\tilde{g}^{n, i}_k(w), 
\end{eqnarray*}
where $\tilde{g}^{n, i}_k(w)$ is a mean zero random function. 
We remark that 
$g^{n, i}_k$, $\bar{g}^{n, i}_k$ and $\tilde{g}^{n, i}_k$ are all 
$\mathcal {B}_{k/n}\otimes \mathcal {F}^n$-measurable. 
The function $f^{n, i}_k(x, w)$ is an additional excess demand associated with the change of the price. 
Usually the higher the stock price becomes, the lower is an agent's demand for the stock. 
Thus it is natural that $f^{n, i}_k(x, w)$ is supposed to be decreasing in $x$. 
Furthermore, we assume
\begin{description}
 \item[\mbox{[A1]}] \ 
The function $f^{n, i}_k(x, w)$ is deterministic (i.e., independent of $\omega \in \Omega ^n$), 
continuous in $(x, w)$, and three times continuously differentiable in $x$. 
Moreover, there exist positive constants $K_0$ and $\delta _0$ such that 
\begin{eqnarray}\label{cond_diff}
-K_0 \leq \frac{\partial }{\partial x}f^{n, i}_k(x, w) \leq  -\delta _0. 
\end{eqnarray}
Furthermore, $\bar{g}^{n, i}_k(w)$ and $\tilde{g}^{n, i}_k(w)$ are continuous in $w$ almost surely. 
\end{description}

Condition [A1] implies that when $I_2$ is not empty (that is, when $N_1 < N$), the equation 
\begin{eqnarray}\label{temp_equi}
\sum _{i\in I}e^{n, i}_k(x, w ; \varphi ^i) = 0
\end{eqnarray}
has a unique solution $x$ for any fixed $w$ and $(\varphi ^i)_i$, since 
the left-hand side of (\ref {temp_equi}) is strictly decreasing in $x$. 
In fact, the existence of a unique solution of (\ref {temp_equi}) is also guaranteed 
even if $I = I_1$, provided $\sum ^N_{i = 1}\Phi ^i > 0$. 
However, hereafter we always assume $N_1 < N$ for some technical reasons which we will explain later. 

By (\ref {market_clearing}) and (\ref {fluctuation_phi}), 
we can construct the processes $(X^n_t)_t$ and $(\varphi ^n_t)_t$. 
We are interested in the limit of the 
$(1+N)$-dimensional process $(\Xi ^n_t)_t = (X^n_t, \varphi ^{n, 1}_t, \ldots , \varphi ^{n, N}_t)_t$ 
as $n\rightarrow \infty $. More precisely, we consider the weak limit of the distribution 
$\mu ^n = P(\Xi ^n\in \cdot )$ on $\mathcal {C}^{1+N}$. 

We will define more conditions. 
\begin{description}
 \item[\mbox{[A2]}] \ For every $M > 0$, there exists a positive constant $C_M > 0$ such that 
\begin{eqnarray*}
\sum ^3_{l = 0}\sup _{|x|, |w|_\infty \leq M}
\left | \frac{\partial ^l}{\partial x^l}f^{n, i}_k(x, w)\right |  + 
\E \hspace{0mm}^n[\sup _{|w|_\infty \leq M}|\bar{g}^{n, i}_k(w)| ^{24}] + 
\E \hspace{0mm}^n[\sup _{|w|_\infty \leq M}|\tilde{g}^{n, i}_k(w)| ^{24}] \leq  C_M 
\end{eqnarray*}
for any $n, k\in \Bbb {N}$ and $i = 1, \ldots , N$, 
where $\E \hspace{0mm}^n$ is the expectation with respect to $P^n$ 
(we simply denote this by $\E $ when there is no possibility of confusion). 
 \item[\mbox{[A3]}] \ The $\sigma $-algebras 
$\sigma (\bar{g}^{n, i}_k, \tilde{g}^{n, i}_k\ ; \ i = 1, \ldots , N)$, $k = 1, 2, \ldots $, 
are independent. 
 \item[\mbox{[A4]}] \ Let 
\begin{eqnarray*}
&&\alpha ^{n, i}_k(w) = - \frac{\partial }{\partial x}f^{n, i}_k(w(k/n), w), \ \ 
\beta ^{n, i}_k(w) = \E \hspace{0mm}^n[\bar{g}^{n, i}_k(w)], \\ 
&&\gamma ^{n, i}_k(w) = \frac{\partial ^2}{\partial x^2}f^{n, i}_k(w(k/n), w), \ \ 
a^{n, ij}_k(w) = \E \hspace{0mm}^n[\tilde{g}^{n, i}_k(w)\tilde{g}^{n, j}_k(w)]. 
\end{eqnarray*}
For each $i, j\in I$ the following limits exist 
\begin{eqnarray*}
&&\beta ^i(t, w) = \lim _{r\rightarrow \infty }\beta ^{n, i}_{[nt]}(w), \ \ 
\gamma ^i(t, w) = \lim _{r\rightarrow \infty }\gamma ^{n, i}_{[nt]}(w), \\
&&a^{ij}(t, w) = \lim _{r\rightarrow \infty }a^{n, ij}_{[nt]}(w)
\end{eqnarray*}
uniformly on any compact subset of $\mathcal {C}$ and for any $t \geq 0$, and 
\begin{eqnarray*}
\alpha ^i(t, w) = \lim _{r\rightarrow \infty }\alpha ^{n, i}_{[nt]}(w)
\end{eqnarray*}
uniformly on any compact subset of $[0, \infty )\times \mathcal {C}$. 
 \item [\mbox{[A5]}] 
Define $Q^n_k(w) = (Q^{n, ij}_k(w))_{1\leq i\leq N, 1\leq j\leq N_1}$ as 
\begin{eqnarray*}
Q^{n, ij}_k(w) = 
(1 - \delta _{ij})\alpha ^{n, i}_k(w)\tilde{\alpha }^{n, j}_k(w), 
\end{eqnarray*}
where $\delta _{ij}$ is the Kronecker delta and 
\begin{eqnarray*}
\bar{\alpha }^n_k(w) = \sum ^N_{i = 1}\alpha ^{n, i}_k(w), \ \ 
\tilde{\alpha }^{n, j}_k(w) = 1 / (\bar{\alpha }^n_k(w) - \alpha ^{n, j}_k(w)). 
\end{eqnarray*}
There exists $V = (V^{ij})^{N_1}_{i, j = 1}\in \Bbb {R}^{N_1}\otimes \Bbb {R}^{N_1}$ 
such that 
$V^{ii} = 0$, $Q^{n, ij}_k(w)\leq V^{ij}$ for each $i, j = 1, \ldots , N_1$ and 
$\rho (V) < 1$, where $\rho (V)$ denotes the spectral radius of $V$. 
\end{description}

Note that the inequality $N_1 < N$ is essential for condition [A5]. 
Indeed, if $N_1 = N$, then the calculation 
\begin{eqnarray*}
\sum ^N_{i=1}((Q^n_k(w))^m)^{ij} &=& 
\sum ^N_{l_{m-1}=1}\cdots \sum ^N_{l_1=1}\sum ^N_{i=1}
Q^{n, il_1}_k(w)Q^{n, l_1l_2}_k(w)\cdots Q^{n, l_{n-1}j}_k(w)\\
&=& 
\sum ^N_{l_{m-1}=1}\cdots \sum ^N_{l_1=1}
Q^{n, l_1l_2}_k(w)\cdots Q^{n, l_{n-1}j}_k(w)
= \cdots = 1, \ \ m\in \Bbb {N}
\end{eqnarray*}
indicates $\rho (Q^n_k(w)) = \lim _{m\rightarrow \infty }||(Q^n_k(w))^m||^{1/m}_1 = 1$, 
where $((Q^n_k(w))^m)^{ij}$ is the $i,j$th element of the $m$th power of the matrix $Q^n_k(w)$ 
and the norm $||\cdot ||_1$ stands for $||A||_1 = \max _j\sum ^N_{i=1}|A^{ij}|$ for 
$A = (A^{ij})_{ij}\in \Bbb {R}^N\otimes \Bbb{R}^N$, 
thus we cannot find such a matrix $V$ in [A5]. 
\begin{description}
 \item[\mbox{[A6]}] \ 
Let $\sigma (t, w) = (\sigma ^{ij}(t, w))^N_{i,j=1}$ be an $N$-dimensional matrix-valued function such that 
$a^{ij}(t, w) = \sum ^N_{m=1}\sigma ^{im}(t, w)\sigma ^{jm}(t, w)$. 
For any $T>0$, there exists a positive constant $C_T$ such that 
\begin{eqnarray*}
|\tilde{\beta }^i(t, w)| + |\tilde {\gamma }^i(t, w)| + |\sigma ^{ij}(t, w)| \leq C_T(1 + |w|_t)
\end{eqnarray*}
for each $i, j = 1, \ldots , N$, $0\leq t\leq T$ and $w\in \Bbb {R}$, 
where 
\begin{eqnarray*}
\tilde{\beta }^i(t, w) = \beta ^i(t, w) + \tilde{\gamma }^i(t, w), \ \ 
\tilde{\gamma }^i(t, w) = \frac{\gamma ^i(t, w)}{2\bar{\alpha }(t, w)^2}\sum ^N_{k, l = 1}a^{kl}(t, w). 
\end{eqnarray*}
\end{description}

We now introduce an SDER.
\begin{eqnarray}\label{SDE_reflection}
dX_t &=& \hat{b}^0(t, X)dt + \sum ^N_{j = 1}\hat{\sigma }^{0j}(t, X)dB^j_t + \sum ^{N_1}_{j = 1}
\tilde{\alpha }^j(t, X)dL^j_t , 
\ \ X_0 = x_0, \\\nonumber 
d\varphi ^i_t &=& \hat{b}^i(t, X)dt + \sum ^N_{j = 1}\hat{\sigma }^{ij}(t, X)dB^j_t + 
1_{I_1}(i)dL^i_t - \sum ^{N_1}_{j = 1}Q^{ij}(t, X)dL^j_t, 
\ \ \varphi ^i_0 = \Phi ^i,\ \ i=1, \ldots , N, 
\end{eqnarray}
where 
\begin{eqnarray*}
\tilde{\alpha }^i(t, w) = 1 / (\bar{\alpha }(t, w) - \alpha ^i(t, w)), \ \ 
Q^{ij}(t, w) = (1 - \delta _{ij})\alpha ^i(t, w)\tilde{\alpha }^j(t, w) 
\end{eqnarray*}
and $\hat{b}^i(t, w)$, $\hat{\sigma }^{ij}(t, w)$, $0\leq i\leq N, 1\leq j\leq N$ are given by 
\begin{eqnarray}\label{def_hat_b}
\hat{b}^i(t, w) &=& \left\{
                  \begin{array}{ll}
                   \sum ^N_{j = 1}\tilde{\beta }^j(t, w) / \bar{\alpha }(t, w)& (i=0)	\\
                   \tilde{\beta }^i(t, w) - \alpha ^i(t, w)\sum ^N_{j = 1}\tilde{\beta }^j(t, w)/\bar{\alpha }(t, w)	& (i\geq 1) , 
                  \end{array}
                  \right. \\\label{def_hat_sigma}
\hat{\sigma }^{ij}(t, w) &=& \left\{
                  \begin{array}{ll}
                   \sum ^N_{k=1}\sigma ^{kj}(t, w)/\bar{\alpha }(t, w)	& (i=0)	\\
                   \sigma ^{ij}(t, w) - \alpha ^i(t, w)\sum ^N_{k=1}\sigma  ^{kj}(t, w)/\bar{\alpha }(t, w)	& (i\geq 1) 
                  \end{array}
                  \right. 
\end{eqnarray}
with $\bar{\alpha }(t, w) = \sum ^N_{i = 1}\alpha ^i(t, w)$. 
We say that an $(1 + N + N_1)$-dimensional continuous adapted stochastic process 
$(X_t, \varphi _t, L_t)_t = (X_t, (\varphi ^i_t)^N_{i=1}, (L^i_t)^{N_1}_{i=1})$ is 
a solution of (\ref {SDE_reflection}) on a given filtered space 
$(\Omega , \mathcal {F}, (\mathcal {F}_t)_t, P)$ 
equipped with an $N$-dimensional $(\mathcal {F}_t)_t$-Brownian motion 
$B_t = (B^i_t)^n_{i=1}$ if 
\begin{itemize}
 \item $P(\varphi \in \mathcal {C}^N_+) = P(L\in \mathcal {C}^{N_1}_{\uparrow 0}) = 1$, 
 \item $\int ^\infty _01_{\{ \varphi ^i_r > 0\} }dL^i_r = 0$ for $i = 1, \ldots , N_1$ 
almost surely, 
 \item The processes $(X_t)_t$, $(\varphi _t)_t$, $(L_t)_t$ and $(B_t)_t$ satisfy 
\begin{eqnarray}\label{SDER_int}
X_t &=& x_0 + \int ^t_0\hat{b}^0(r, X)dr + \sum ^N_{j=1}\int ^t_0\hat{\sigma }^{0j}(r, X)dB^j_r + 
\sum ^{N_1}_{j = 1}\int ^t_0\tilde{\alpha }^j(r, X)dL^j_r, \\\nonumber 
\varphi ^i_t &=& \Phi ^i + \int ^t_0\hat{b}^i(r, X)dr - \sum ^N_{j = 1}\int ^t_0\hat{\sigma }^{ij}(r, X)dB^j_r + 
1_{I_1}(i)L^i_t - \sum ^{N_1}_{j = 1}\int ^t_0Q^{ij}(r, X)dL^j_r
\end{eqnarray}
for $t\geq 0$ and $i = 1, \ldots , N$ almost surely. 
\end{itemize}
We call the process $L = (L^i_t)_{i\in I_1, t\geq 0}$ 
a regulator associated with $\Xi = (X, \varphi )$. 
Now we present our final assumption. 

\begin{description}
 \item[\mbox{[A7]}] \ A solution of (\ref {SDE_reflection}) is unique in law. 
\end{description}

For instance, if $\partial f^n_k/\partial x$, $i = 1, \ldots , N$, 
is constant (and so is the matrix $Q$), condition [A7] holds 
under a Lipschitz condition on the coefficients $\hat{b}^i$ and $\hat{\sigma }^{ij}$ 
(see \cite {Czarkowski-Slominski} 
for instance:
although the form of our SDE (\ref {SDE_reflection}) is a little special, 
the arguments in these papers also works.) 
For other sufficient conditions for [A7], see \cite {Dupuis-Ishii2} and \cite {Piera-Mazumdar}. 

By [A7], the distributions $\mu = P(\Xi \in \cdot )$ and $\nu = P((\Xi , L)\in \cdot  )$ are unique if exist. 
We are now prepared to state our main result. 

\begin{thm} \ \label{th_converge}Assume $[A1]$--$[A7]$. 
Then $\mu $ and $\nu $ exist and the distribution $\mu ^n$ converges weakly to $\mu $ on 
$\mathcal {C}^{1+N}$ as $n\rightarrow \infty $. 
\end{thm}

The proof is in Section \ref {sec_proof_th_converge}. 
Theorem \ref {th_converge} implies that the limit of 
$\Xi ^n = (X^n, \varphi ^{n, 1}, \ldots , \varphi ^{n, N})$ 
is characterized as the solution of an SDER, viz., (\ref {SDE_reflection}). 
The regulator process $L$ prevents the shares of stock of an agent $i$ 
from taking a negative value. 
The infinitesimal term $1_{I_1}(i)dL^i_t$ in (\ref {SDE_reflection}) works only when 
the agent $i\in I_1$ hopes to sell more  of the stock than they hold.

Here we consider the case where all agents in the market can sell short
(i.e., $N_1 = 0$). 
This is a version of Theorem 1 of \cite {Kato} and 
$X^n$ converges weakly to the unique solution $\hat{X}$ of 
\begin{eqnarray}\label{SDE_1dim}
d\hat{X}_t = \hat{b}^0(t, \hat{X})dt + \sum ^N_{j=1}\hat{\sigma }^{0j}(t, \hat{X})dB^j_t, 
\ \ \hat{X}_0 = x_0. 
\end{eqnarray}
In this case, no agents are bound by the short sale prohibition, 
and the process $(\hat{X}_t)_t$ represents the log-price of the stock without friction. 

On the other hand, when $I_1$ is not empty, 
agents in $I_1$ may not be able to exhibit their primary excess demand 
(which is described as $\tilde{e}^{n, i}_k$ in the discrete-time model.) 
The (log-)price $X_t$ is pushed up by the gap between the actual excess demand 
with the primary excess demand, so $X_t$ is larger than $\hat{X}_t$. 
To describe such a phenomenon, we give the following additional assumption. 

\begin{description}
 \item[\mbox{[A8]}] \ There exist measurable functions $\tilde{\sigma }^j(t, x)$, $j = 1, \ldots , N$, 
such that $\hat{\sigma }^{0j}(t, w) = \tilde{\sigma }^j(t, w(t))$ for each $t$ and 
\begin{eqnarray*}
\sum ^N_{j = 1}|\tilde{\sigma }^j(t, x) - \tilde{\sigma }^j(t, y)| \leq \rho (|x-y|), \ \ 
t\geq 0, \ x, y\in \Bbb {R}^N
\end{eqnarray*}
for some strictly increasing continuous function $\rho : [0, \infty )\longrightarrow [0, \infty )$ 
satisfying $\rho (0) = 0$ and 
\begin{eqnarray*}
\int _{(0, \infty )}\frac{1}{\rho (\xi )^2}d\xi  < \infty . 
\end{eqnarray*} 
 \item[\mbox{[A9]}] \ The pathwise uniqueness of solutions of (\ref {SDE_1dim}) holds. 
\end{description}

\begin{thm} \ \label{th_comparison}Assume $[A1]$--$[A9]$. 
Let $(X, \varphi , L)$ 
(resp., $\hat{X}$) be a solution of (\ref {SDE_reflection}) (resp., (\ref {SDE_1dim})) 
on a given filtered space $(\Omega , \mathcal {F}, (\mathcal {F}_t)_t, P)$ equipped with an $N$-dimensional Brownian motion $B$. 
Then $X_t\geq \hat{X}_t$ for any $t\geq 0$ almost surely. 
\end{thm}

Theorem \ref {th_comparison} is 
directly obtained by the same proof as Theorem 1.1 in \cite {Ikeda-Watanabe}. 
This suggests the assertion that the short sale constraint causes the overpricing in our model. 

\subsection{Model II: Stock Price Model under Budget Constraint
}\label{subsec_no_borrow}
We also consider the case where some of the agents, $I_1 = \{1, \ldots , N_1\}$, cannot borrow cash. 
In the previous section, an excess demand function with no friction 
$\tilde{e}^{n, i}_k(x, w)$ is understood as shares of the stock which an agent wants to buy. 
In this section, we interpret $\tilde{e}^{n, i}_k(x, w)$ to mean 
an excess demand in terms of dollars. We also define 
\begin{eqnarray*}
e^{n, i}_k(x, w  ; y ) = 
\left\{
\begin{array}{ll}
 \min \{ \tilde {e}^{n, i}_k(x, w ), y \} , 	& i\in I_1	\\
 \tilde {e}^{n, i}_k(x, w ), 	& i\in I_2. 
\end{array}
\right.
\end{eqnarray*}
The market clearing condition is now expressed by 
\begin{eqnarray}\label{market_clearing2}
\sum ^N_{i = 1}e^{n, i}_k(X^n_{(k+1)/n}, X^n ; W^n_{k/n} ) = 0, 
\end{eqnarray}
where $W^n_t$ is the agent's amount of cash held at time $t$. 
Then the process of the log-price of the stock $(X^n_t)_t$ 
and the cash holdings $(W^{n, i}_t)_t$ of an agent $i$ are given by 
(\ref {linear_interpolation}), (\ref {market_clearing2}), 
\begin{eqnarray}\label{fluctuation_W}
W^{n, i}_{(k+1)/n} = W^{n, i}_{k/n} - e^{n, i}_k(X^n_{(k+1)/n}, X^n ; W^{n, i}_{k/n}),
\end{eqnarray}
and 
\begin{eqnarray}\label{linear_interpolation_W}
W^{n, i}_t = (nt - k)W^{n, i}_{(k+1)/n} + (k + 1 - nt)W^{n, i}_{k/n},\ \ t\in (k/n, (k+1)/n). 
\end{eqnarray}
We assume that $\tilde {e}^{n, i}_k$ has the same form as (\ref {form_e}).
We also assume [A1]--[A7],
replacing $\hat{b}$, $\hat{\sigma }$, and (\ref {SDE_reflection}) with 
\begin{eqnarray*}
\hat{b}^i(t, w) &=& \left\{
                  \begin{array}{ll}
                   \sum ^N_{j = 1}\tilde{\beta }^j(t, w) / \bar{\alpha }(t, w)& (i=0)	\\
                   -\tilde{\beta }^i(t, w) + \alpha ^i(t, w)\sum ^N_{j = 1}\tilde{\beta }^j(t, w)/\bar{\alpha }(t, w)	& (i\geq 1) , 
                  \end{array}
                  \right. \\
\hat{\sigma }^{ij}(t, w) &=& \left\{
                  \begin{array}{ll}
                   \sum ^N_{k=1}\sigma ^{kj}(t, w)/\bar{\alpha }(t, w)	& (i=0)	\\
                   -\sigma ^{ij}(t, w) + \alpha ^i(t, w)\sum ^N_{k=1}\sigma  ^{kj}(t, w)/\bar{\alpha }(t, w)	& (i\geq 1) 
                  \end{array},
                  \right. 
\end{eqnarray*}
and 
\begin{eqnarray}\label{SDE_reflection2}
dX_t &=& \hat{b}^0(t, X)dt + \sum ^N_{j = 1}\hat{\sigma }^{0j}(t, X)dB^j_t - \sum ^{N_1}_{j = 1}
\tilde{\alpha }^j(t, X)dL^j_t , 
\ \ X_0 = x_0, \\\nonumber 
dW^i_t &=& \hat{b}^i(t, X)dt + \sum ^N_{j = 1}\hat{\sigma }^{ij}(t, X)dB^j_t + 
1_{I_1}(i)dL^i_t + \sum ^{N_1}_{j = 1}Q^{ij}(t, X)dL^j_t, 
\ \ W^i_0 = c^i,\ \ i=1, \ldots , N, 
\end{eqnarray}
where $c^i\geq 0$ is the initial cash holdings of agent $i$. 
Then we have the following theorem. 

\begin{thm} \ \label{th_converge2}
The distribution of $\tilde{\Xi }^n = (X^n, W^{n, 1}, \ldots , W^{n, N})$ 
converges weakly to a solution of $(\ref {SDE_reflection2})$ on 
$\mathcal {C}^{1+N}$ as $n\rightarrow \infty $. 
\end{thm}

\begin{thm} \ \label{th_comparison2}Let $\tilde{\Xi }= (X, W^1, \ldots , W^N)$ 
(resp., $\hat{X}$) be a solution of $(\ref {SDE_reflection2})$ (resp., (\ref {SDE_1dim})) 
on a given filtered space $(\Omega , \mathcal {F}, (\mathcal {F}_t)_t, P)$ equipped with an $N$-dimensional Brownian motion $B$. 
Then $X_t\leq \hat{X}_t$ for any $t\geq 0$ almost surely. 
\end{thm}

We omit the proofs of Theorems \ref {th_converge2}--\ref {th_comparison2} 
since they are almost the same as those of Theorems \ref {th_converge}--\ref {th_comparison}.

\section{Proof of Theorem \ref {th_converge}}\label{sec_proof_th_converge}

Take any $M > |x_0| $ and let $\psi _M\in C^{\infty }(\Bbb {R} ; [0, 1])$ be such that 
$\psi _M(y) = 1$ on $|y|\leq M/2$ and $\psi _M(y) = 0$ on $|y| \geq M$. 
We set 
\begin{eqnarray*}
\tilde{e}^{n, M, i}_k(x, w) = 
-(1-\psi _M(w(k/n)))\alpha ^{n, i}_k(w)(x - w(k/n)) + \psi _M(w(k/n))\tilde{e}^{n, i}_k(x, w). 
\end{eqnarray*}
We define $X^{n, M}$, $\varphi ^{n, M, i}$ and $e^{n, M, i}_k(x, w ; \varphi )$ similarly to 
(\ref {market_clearing})--(\ref {linear_interpolation_phi}),
replacing $\tilde{e}^{n, i}_k$ with $\tilde{e}^{n, M, i}_k$. 
First we consider the convergence of the truncated processes 
$\Xi ^{n, M} = (X^{n, M}, \varphi ^{n, M, 1}, \ldots , \varphi ^{n, M, N})$, $n\in \Bbb {N}$ 
for fixed $M$. 
We can easily see the following proposition (Proposition 1 in \cite {Kato}). 

\begin{prop} \ \label{prop_range_XM}
For any $\omega $, if $|X^{n, M}_t(\omega )|\leq M$, 
then $|X^{n, M}_r(\omega )|\leq M$ for all $r\in [0, t]$. 
\end{prop}

We rearrange our market clearing equation into the form of a difference equation. 
Since it follows that 
\begin{eqnarray*}
\max \{ \tilde{e}^{n, M, i}_k(x, w), -\varphi \} = 
\tilde{e}^{n, M, i}_k(x, w) + (-\tilde{e}^{n, M, i}_k(x, w) - \varphi )_+, 
\end{eqnarray*}
we get 
\begin{eqnarray*}
\sum ^N_{i = 1}e^{n, M, i}_k(X^{n, M}_{(k+1)/n}, X^{n, M} ; \varphi ^{n, M, i}_{k/n}) = 
\sum ^N_{i = 1}\tilde {e}^{n, M, i}_k(X^{n, M}_{(k+1)/n}, X^{n, M}) + 
\sum ^{N_1}_{i = 1}\hat{\eta }^{n, M, i}_k = 0, 
\end{eqnarray*}
where $\hat{\eta }^{n, M, i}_k = (-\tilde {e}^{n, M, i}_k(X^{n, M}_{(k+1)/n}, X^{n, M}) - \varphi ^{n, M, i}_{k/n})_+$. 
Using Taylor's theorem, we have 
\begin{eqnarray}\nonumber 
&&X^{n, M}_{(k+1)/n} - X^{n, M}_{k/n}\\\nonumber  &=& 
\frac{1}{\bar{\alpha }^n_k(X^{n, M})}\Big\{ 
\psi _M(X^{n, M}_{k/n})\sum ^N_{i=1}\left ( g^{n, i}_k(X^{n, M}) + 
\frac{1}{2}\gamma ^{n, i}_k(X^{n, M})(X^{n, M}_{(k+1)/n} - 
X^{n, M}_{k/n})^2 + \varepsilon ^{n, M, i}_k \right )\\\label{diff_X}&&\hspace{20mm} + 
\sum ^{N_1}_{i=1}\hat{\eta }^{n, M, i}_k \Big\} 
\end{eqnarray}
and 
\begin{eqnarray}\nonumber &&
\varphi ^{n, M, i}_{(k+1)/n} - \varphi ^{n, M, i}_{k/n}\\\nonumber 
&=& \psi _M(X^{n, M}_{k/n})\left( g^{n, i}_k(X^{n, M}) + 
\frac{1}{2}\gamma ^{n, i}_k(X^{n, M})(X^{n, M}_{(k+1)/n} - 
X^{n, M}_{k/n})^2 + \varepsilon ^{n, M, i}_k\right)\\\label{diff_phi} && - 
\alpha ^{n, i}_k(X^{n, M})(X^{n, M}_{(k+1)/n} - X^{n, M}_{k/n})  + 1_{I_1}(i)\hat{\eta }^{n, M, i}_k, 
\end{eqnarray}
where 
\begin{eqnarray*}
\varepsilon ^{n, M, i}_k = \frac{1}{2}\int ^1_0(1-u)^2\frac{\partial ^3}{\partial x^3}f^{n, i}_k
(uX^{n, M}_{(k+1)/n} + (1-u)X^{n, M}_{k/n}, X^{n, M})du
(X^{n, M}_{(k+1)/n} - X^{n, M}_{k/n})^3. 
\end{eqnarray*}
Substituting (\ref {diff_X}) into itself and into (\ref {diff_phi}), we get 
\begin{eqnarray*}
X^{n, M}_{(k+1)/n} - X^{n, M}_{k/n} &=& 
\frac{1}{\bar{\alpha }^n_k(X^{n, M})}\left\{ 
\sum ^N_{i = 1}H^{n, M, i}_k(X^{n, M}) + \sum ^{N_1}_{i = 1}\hat{\eta }^{n, M, i}_k\right\} , \\
\varphi ^{n, M, i}_{(k+1)/n} - \varphi ^{n, M, i}_{k/n} &=& 
H^{n, M, i}_k(X^{n, M}) - \frac{\alpha ^{n, i}_k(X^{n, M})}{\bar{\alpha }^n_k(X^{n, M})}
\left\{ 
\sum ^N_{j = 1}H^{n, M, j}_k(X^{n, M}) + \sum ^{N_1}_{j = 1}\hat{\eta }^{n, M, j}_k\right\} \\&& + 
1_{I_1}(i)\hat{\eta }^{n, M, i}_k, 
\end{eqnarray*}
where 
\begin{eqnarray*}
H^{n, M, i}_k(w) &=& 
\psi _M(w(k/n))\left\{ \frac{1}{\sqrt{n}}\tilde{g}^{n, i}_k(w) + \frac{1}{n}h^{n, i}_k(w)\right\} + 
\tilde{\varepsilon }^{n, M, i}_k, \\
h^{n, i}_k(w) &=& \bar{g}^{n, i}_k(w) + \psi _M(w(k/n))^2\frac{\gamma ^{n, i}_k(w)}{2\bar{\alpha }^n_k(w)}
\sum ^{N}_{j, m = 1}\tilde{g}^{n, j}_k(w)\tilde{g}^{n, m}_k(w), \\
\tilde{\varepsilon }^{n, M, i}_k &=& \psi _M(X^{n, M}_{k/n})\left\{ \varepsilon ^{n, M, i}_k + 
\frac{\gamma ^{n, i}_k(X^{n, M})}{2\bar{\alpha }^n_k(X^{n, M})^2}\hat{\varepsilon }^{n, M}_k\right\} , \\
\hat{\varepsilon }^{n, M}_k &=& 
\left( \sum ^{N_1}_{i = 1}\hat{\eta }^{n, M, i}_k\right) \left\{ \sum ^{N_1}_{i = 1}\hat{\eta }^{n, M, i}_k + 
2\psi _M(X^{n, M}_{k/n})\sum ^N_{i = 1}\left (\frac{1}{\sqrt{n}}\tilde{g}^{n, i}_k(X^{n, M}) + \rho ^{n, M, i}_k\right )
\right\} \\&& + 
\psi _M(X^{n, M}_{k/n})^2\left\{ \frac{2}{\sqrt{n}}
\sum ^N_{i, j = 1}\tilde{g}^i_k(X^{n, M})\rho ^{n, M, j}_k + 
\left (\sum ^N_{i = 1}\rho ^{n, M, i}_k\right )^2\right\} , \\
\rho ^{n, M, i}_k &=& \frac{1}{n}\bar{g}^{n, i}_k(X^{n, M}) + 
\frac{1}{2}\gamma ^{n, i}_k(X^{n, M})(X^{n, M}_{(k+1)/n} - 
X^{n, M}_{k/n})^2 + \varepsilon ^{n, M, i}_k. 
\end{eqnarray*}
Thus, if we set 
$\eta ^{n, M, i}_k = (1 - \alpha ^{n, i}_k(w)/\bar {\alpha }^n_k(w))\hat{\eta }^{n, M, i}_k$, 
\begin{eqnarray*}
Z^{n, M, i}_t = \sum ^{[nt] - 1}_{k = 0}H^{n, M, i}_k(X^{n, M}) + (nt - [nt])H^{n, M, i}_{[nt]}(X^{n, M}), \ \ 
L^{n, M, i}_t = \sum ^{[nt] - 1}_{k = 0}\eta ^{n, M, i}_k + (nt - [nt])\eta ^{n, M, i}_{[nt]} 
\end{eqnarray*}
and 
\begin{eqnarray}\label{def_Yn1}
Y^{n, M, 0}_t &=& x_0 + 
\sum ^N_{i=1}\int ^t_0\frac{1}{\bar{\alpha }^n_{[nr]}(X^{n, M})}dZ^{n, M, i}_r, \\\label{def_Yn2}
Y^{n, M, i}_t &=& \Phi ^i + Z^{n, M, i}_t - \sum ^N_{j = 1}
\int ^t_0\frac{\alpha ^{n, i}_{[nr]}(X^{n, M})}{\bar{\alpha }^n_{[nr]}(X^{n, M})}dZ^{n, M, j}_r, \ \ i = 1, \ldots , N, 
\end{eqnarray}
then
\begin{eqnarray}\label{SP_like_X}
X^{n, M}_t &=& Y^{n, M, 0}_t + 
\sum ^{N_1}_{i=1}\int ^t_0\tilde{\alpha }^{n, i}_{[nr]}(X^{n, M})dL^{n, M, i}_r, \\\label{SP_like}
\varphi ^{n, M, i}_t &=& Y^{n, M, i}_t + 
1_{I_1}(i)L^{n, M, i}_t - \sum ^{N_1}_{j = 1}\int ^t_0Q^{n, ij}_{[nr]}(X^{n, M})dL^{n, M, j}_t. 
\end{eqnarray}

The equality (\ref {SP_like}) seems to imply that 
$(\varphi ^{n, M, i}, L^{n, M, i})^{N_1}_{i = 1}$ 
is a solution of the Skorokhod problem with oblique reflection in the non-negative orthant 
associated with $(Y^{n, M, i})^{N_1}_{i = 1}$ 
(see \cite {Czarkowski-Slominski}, \cite{Harrison-Reiman}, 
\cite{Piera-Mazumdar}, \cite {Ramasubramanian} and \cite {Shashiashvili}).
However this is not strictly true, 
since the equality $\int ^\infty _0\varphi ^{n, M, i}_rdL^{n, i}_r = 0$ does not hold 
by virtue of the linear interpolation (\ref {linear_interpolation_phi}). 
The following proposition tells us that 
$(\varphi ^{n, M, i}_{k/n}, L^{n, M, i}_{k/n})^{N_1}_{i = 1}$, $k\in \Bbb {Z}_+$, is, as it were, 
a solution of the corresponding Skorokhod problem in discrete-time. 

\begin{prop} \ \label{prop_discrete_SP}For every $k\in \Bbb {Z}_+$, 
\begin{eqnarray}\label{eq_discrete_SP}
L^{n, M, i}_{k/n} = \max _{0\leq l\leq k}\left(\sum ^{N_1}_{j = 1}
\int ^{l/n}_0Q^{n, ij}_{[nr]}(X^{n, M})dL^{n, M, j}_r - Y^{n, M, i}_{l/n}\right) _+. 
\end{eqnarray}
\end{prop}

\begin{proof} 
It is obvious that the left-hand side of (\ref {eq_discrete_SP}) is not less than the right-hand side. 
We suppose 
\begin{eqnarray}\label{ineq_discrete_SP}
L^{n, M, i}_{k/n}> \left(\sum ^{N_1}_{j = 1}
\int ^{l/n}_0Q^{n, ij}_{[nr]}(X^{n, M})dL^{n, M, j}_r - Y^{n, M, i}_{l/n}\right) _+
\end{eqnarray}
for $l = 0. \ldots , k$. 
By (\ref {ineq_discrete_SP}) with $l = k$, we have 
\begin{eqnarray*}
Y^{n, M, i}_{k/n} + L^{n, M, i}_{k/n} - \int ^{k/n}_0Q^{n, ij}_{[nr]}(X^{n, M})dL^{n, M, j}_r = 
\varphi ^{n, M, i}_{k/n} > 0. 
\end{eqnarray*}
This inequality gives $\eta ^{n, M, i}_{k-1} = 0$ 
(by the definition of $\hat{\eta }^{n, M, i}_{k-1}$) 
and thus $L^{n, M, i}_{(k-1)/n} = L^{n, M, i}_{k/n}$. 
Using (\ref {ineq_discrete_SP}) again with $l = k - 1$, 
we similarly get $\eta ^{n, M, i}_{k-2} = 0$. 
Inductively we see that $L^{n, M, i}_t = 0$ for $t\in [0, k/n]$ and 
this contradicts (\ref {ineq_discrete_SP}). 
Then we obtain the assertion. 
\end{proof}

Using the above proposition and the same arguments as Theorem 2 in \cite {Shashiashvili}, 
we get the following proposition. 

\begin{prop} \ \label{prop_discrete_diff_l}For every $0 \leq l \leq k$, 
\begin{eqnarray*}
\sum ^{N_1}_{i = 1}|L^{n, M, i}_{k/n} - L^{n, M, i}_{l/n}|^2 \leq 
\hat{K}\max _{l\leq m\leq k}\sum ^{N_1}_{i = 1}|Y^{n, M, i}_{m/n} - Y^{n, M, i}_{l/n}|^2
\end{eqnarray*}
for some $\hat{K} > 0$ depending only on $V$. 
\end{prop}

The equality (\ref {diff_phi}) also indicates 
\begin{eqnarray}\label{discrete_classical_SP}
\varphi ^{n, M, i}_{k/n} = \hat{Y}^{n, M, i}_k + \hat{L}^{n, M, i}_{k/n}, i = 1, \ldots , N_1, 
\end{eqnarray}
where 
\begin{eqnarray*}
\hat{Y}^{n, M, i}_k &=& 
\sum ^{k-1}_{l = 0}\Bigg\{ 
\psi _M(X^{n, M}_{l/n})\left( g^{n, i}_l(X^{n, M}) + 
\frac{1}{2}\gamma ^{n, i}_l(X^{n, M})(X^{n, M}_{(l+1)/n} - 
X^{n, M}_{l/n})^2 + \varepsilon ^{n, M, i}_l\right)\\&&\hspace{10mm} - 
\alpha ^{n, i}_l(X^{n, M})(X^{n, M}_{(l+1)/n} - X^{n, M}_{l/n})\Bigg\} , \\
\hat{L}^{n, M, i}_t &=& \sum ^{[nt] - 1}_{k = 0}\hat{\eta }^{n, M, i}_k + (nt - [nt])\hat{\eta }^{n, M, i}_{[nt]}. 
\end{eqnarray*}
The equality (\ref {discrete_classical_SP}) corresponds to the classical 
Skorokhod problem for each $i = 1, \ldots , N_1$. 
Similarly to Proposition \ref {prop_discrete_SP}, we obtain the following. 

\begin{prop} \ \label{prop_discrete_SP2}For every $k\in \Bbb {Z}_+$, 
$\hat{L}^{n, M, i}_{k/n} = \max _{0\leq l\leq k}(-\hat{Y}^{n, M, i}_{l/n})_+$. 
\end{prop}

Next we evaluate the moment of the process. 

\begin{prop} \ \label{prop_moment_diff_X}
$\E [|X^{n, M}_{(k+1)/n} - X^{n, M}_{k/n}|^{24}]\leq C_M / n^{12}$, \ $k\in \Bbb {Z}_+$ for some $C_M > 0$. 
\end{prop}

\begin{proof} 
Set 
\begin{eqnarray*}
f(t) = \tilde {f}(t(X^{n, M}_{(k+1)/n} - X^{n, M}_{k/n}) + X^{n, M}_{k/n}), \ \ 
\tilde {f}(x) = \sum ^N_{i = 1}e^{n, M, i}_k(x, X^{n, M} ; \varphi ^{n, M, i}_{k/n}). 
\end{eqnarray*}
Then we have $f(1) = 0$ and
\begin{eqnarray*}
f(0) = \sum ^{N_1}_{i = 1}\max \{ \psi _M(X^{n, M}_{k/n})g^{n, i}_k(X^{n, M}), -\varphi ^{n, M, i}_{k-1} \}  + 
\sum ^{N}_{i = N_1 + 1}\psi _M(X^{n, M}_{k/n})g^{n, i}_k(X^{n, M}). 
\end{eqnarray*}
Using the fundamental theorem of calculus, we get 
\begin{eqnarray}
f(0) = f(0) - f(1) = -\int ^1_0\tilde {f}'(t(X^{n, M}_{(k+1)/n} - X^{n, M}_{k/n}) + X^{n, M}_{k/n})
dt(X^{n, M}_{(k+1)/n} - X^{n, M}_{k/n}). 
\end{eqnarray}
Since $\varphi ^{n, M, i}_t\geq 0$ for $i = 1, \ldots , N_1$ and 
$\tilde{f}'(x) \leq -N_2\delta _0$, we get 
\begin{eqnarray}\label{temp_moment1}
|X^{n, M}_{(k+1)/n} - X^{n, M}_{k/n}| \leq \frac{|f(0)|}{N_2\delta _0}
\leq 
\frac{1}{N_2\delta _0}\sum ^N_{i=1}|\psi _M(X^{n, M}_{k/n})g^{n, i}_k(X^{n, M})|. 
\end{eqnarray}
By [A2] and Proposition \ref {prop_range_XM}, we have 
\begin{eqnarray}\label{temp_moment2}
\E [\psi _M(X^{n, M}_{k/n})|g^{n, i}_k(X^{n, M})|^{24}] = 
\E [\sup _{|w|_\infty \leq M}|g^{n, i}_k(y)|^{24}] \leq \frac{C_M}{n^{12}}. 
\end{eqnarray}
Our assertion follows immediately by (\ref {temp_moment1}) and (\ref {temp_moment2}). 
\end{proof}

The above proposition and [A1]--[A2] lead us to the following. 

\begin{prop} \ \label{prop_moment_eta}
$\E [|\tilde {e}^{n, M, i}_k(X^{n, M}_{(k+1)/n}, X^{n, M})|^{24}] + 
\E [|\eta ^{n, M, i}_k|^{24}] + \E [|\hat{\eta }^{n, M, i}_k|^{24}] + 
\E [|\varepsilon ^{n, M, i}_k|^8]\leq C_M/n^{12}$, \ $k\in \Bbb {Z}_+$, 
for some $C_M > 0$. 
\end{prop}

\begin{prop} \ \label{prop_bdd_L}
$\E [|L^{n, M}_t|^8] + \E [|\hat{L}^{n, M}_t|^8] \leq C_{M, t}$ for some $C_{M, t} > 0$. 
\end{prop}

\begin{proof} 
It suffices to estimate 
$\E [|\sum ^{N_1}_{i = 1}\hat{L}^{n, M, i}_{[nt]/n}|^8]$. 
By [A1], Proposition \ref{prop_discrete_SP2}, and the equality 
\begin{eqnarray}\label{eq_indicator}
X^{n, M}_{(k+1)/n} - X^{n, M}_{k/n} = (X^{n, M}_{(k+1)/n} - X^{n, M}_{k/n})1_{\{|X^{n, M}_{k/n}|\leq M\}}
\end{eqnarray}
(obtained by the definitions of $\tilde {e}^{n, M, i}_k$, $\hat{\eta }^{n, M, i}_k$, and 
$X^{n, M}_t$), we have 
\begin{eqnarray*}
\sum ^{N_1}_{i = 1}\hat{L}^{n, M, i}_{[nt]/n} &\leq & 
\sum ^{N_1}_{i = 1}\max _{0\leq k\leq [nt]}\Bigg| 
G^{n, M, i}_k - \tilde{G}^{n, M, i}_k + 
\frac{1}{n}\sum ^{k-1}_{l = 0}\left( \hat{h}^{n, M, i}_l - 
\frac{\alpha ^{n, i}_l(X^{n, M})}{\bar{\alpha }^n_l(X^{n, M})}
\sum ^N_{j = 1}\hat{h}^{n, M, j}_l \right)\\&&\hspace{20mm}  - 
\sum ^{k-1}_{l = 0}\frac{\alpha ^{n, i}_l(X^{n, M})}{\bar{\alpha }^n_l(X^{n, M})}
\sum ^{N_1}_{j = 1}\hat{\eta }^{n, M, j}_l
\Bigg|\\
&\leq & 
\sum ^{N_1}_{i = 1}\left( \max _{0\leq k\leq [nt]}|G^{n, M, i}_k| + 
\max _{0\leq k\leq [nt]}|\tilde{G}^{n, M, i}_k|\right) + 
\frac{2}{n}\sum ^N_{i = 1}\sum ^{[nt]-1}_{k = 0}|\hat{h}^{n, M, i}_k| \\&& + 
\sum ^{[nt]-1}_{k = 0}
\frac{\sum ^{N_1}_{i = 1}\alpha ^{n, i}_k(X^{n, M})}{\bar{\alpha }^n_k(X^{n, M})}
\sum ^{N_1}_{j = 1}\hat{\eta }^{n, M, j}_k, 
\end{eqnarray*}
where 
\begin{eqnarray*}
G^{n, M, i}_k &=& \frac{1}{\sqrt{n}}\sum ^{k - 1}_{l = 0}\psi _M(X^{n, M}_{l/n})\tilde{g}^{n, i}_l(X^{n, M}), \\
\tilde{G}^{n, M, i}_k &=& 
\frac{1}{\sqrt{n}}\sum ^N_{j = 1}\sum ^{k - 1}_{l = 0}\frac{\alpha ^{n, i}_l(X^{n, M})}{\bar{\alpha }^n_l(X^{n, M})}
\psi _M(X^{n, M}_{l/n})\tilde{g}^{n, j}_l(X^{n, M}), \\
\hat{h}^{n, M, i}_k &=& \psi _M(X^{n, M}_{k/ n})\left( 
\bar{g}^{n, i}_k + \frac{n}{2}\gamma ^{n, i}_k(X^{n, M})(X^{n, M}_{(k+1)/n} - X^{n, M}_{k/n})^2 + 
n\varepsilon ^{n, M, i}_k\right) . 
\end{eqnarray*}
Since 
\begin{eqnarray*}
\frac{\sum ^{N_1}_{i = 1}\alpha ^{n, i}_k(X^{n, M})}{\bar{\alpha }^n_k(X^{n, M})} \leq 
\frac{1}{1 + N_2\delta _0/ (N_1K_0)}
\end{eqnarray*}
by virtue of [A1], we get 
\begin{eqnarray*}
\sum ^{N_1}_{i = 1}\hat{L}^{n, M, i}_{[nt]/n} &\leq & 
\sum ^{N_1}_{i = 1}\left( \max _{0\leq k\leq [nt]}|G^{n, M, i}_k| + 
\max _{0\leq k\leq [nt]}|\tilde{G}^{n, M, i}_k|\right) + 
\frac{2}{n}\sum ^N_{i = 1}\sum ^{[nt]-1}_{k = 0}|\hat{h}^{n, M, i}_k|\\&& + 
\frac{1}{1 + N_2\delta _0/ (N_1K_0)}\sum ^{N_1}_{i = 1}\hat{L}^{n, M, i}_{[nt]/n},
\end{eqnarray*}
and thus 
\begin{eqnarray}\nonumber 
&&\sum ^{N_1}_{i = 1}\hat{L}^{n, M, i}_{[nt]/n}\\\label{temp_diff_L_1}&\leq & 
\left( 1 + \frac{N_1K_0}{N_2\delta _0}\right) \left\{ 
\sum ^{N_1}_{i = 1}\left( \max _{0\leq k\leq [nt]}|G^{n, M, i}_k| + 
\max _{0\leq k\leq [nt]}|\tilde{G}^{n, M, i}_k|\right) + 
\frac{2}{n}\sum ^N_{i = 1}\sum ^{[nt]-1}_{k = 0}|\hat{h}^{n, M, i}_k|\right\} . \hspace{15mm}
\end{eqnarray}
By [A4] and Propositions \ref {prop_moment_diff_X}--\ref {prop_moment_eta}, 
we have 
\begin{eqnarray}\label{temp_diff_L_2}
\sum ^N_{i = 1}\E \left [\left (\sum ^{[nt]-1}_{k = 0}|\hat{h}^{n, M, i}_k|\right )^8\right ] 
\leq C'_{M, t}
\end{eqnarray}
for some $C'_{M, t} > 0$. 
Moreover, since $(G^{n, M, i}_k)_k$ and $(\tilde{G}^{n, M, i}_k)_k$ are both 
$(\mathcal {G}^n_k)_k$-martingales, where 
$\mathcal {G}^n_k$ is the $\sigma $-algebra generated by 
$\bar{g}^{n, i}_l$ and $\tilde{g}^{n, i}_l$ for $i = 1, \ldots , N$, $l = 0,\ \ldots , k-1$ 
(note that $\mathcal {G}^n_0$ is a trivial $\sigma $-algebra), 
the Doob inequality implies 
\begin{eqnarray}\nonumber 
&&\E [\max _{0\leq k\leq [nt]}|G^{n, M, i}_k|^8 + \max _{0\leq k\leq [nt]}|\tilde{G}^{n, M, i}_k|^8]\leq  
\left(\frac{8}{7}\right)^8 
\E [|G^{n, M, i}_{[nt]}|^8 + |\tilde{G}^{n, M, i}_{[nt]}|^8]\\\label{temp_diff_L_3}
&&\leq 
\frac{2}{n^2}\left(\frac{8}{7}\right)^8 
\E \left[\left (\sum ^{[nt]-1}_{k = 0}\psi _M(X^{n, M}_{k/n})
\tilde{g}^{n, i}_k(X^{n, M})\right) ^8\right ]
\leq C''_{M, t}
\end{eqnarray}
for some $C''_{M, t} > 0$. 
Now by (\ref {temp_diff_L_1})--(\ref {temp_diff_L_3}), we obtain the assertion. 
\end{proof}

\begin{prop} \ \label{prop_moment_eps}For every $t > 0$ there exists a constant $C_{M, t} > 0$ such that 
\begin{eqnarray*}
\E \Big [\Big (\sum ^{[nt] - 1}_{k = 0}|\tilde{\varepsilon }^{n, M, i}_k|\Big) ^4\Big] \leq 
\frac{C_{M, t}}{n}. 
\end{eqnarray*}
\end{prop}

\begin{proof} 
It suffices to show this for $\hat{\varepsilon }^{n, M}_k$ instead of $\tilde{\varepsilon }^{n, M, i}_k$. 
A straightforward calculation gives 
\begin{eqnarray*}
\E \Big [\Big (\sum ^{[nt] - 1}_{k = 0}|\hat{\varepsilon }^{n, M}_k|\Big) ^4\Big] &\leq & 
\E \Big [\Big \{ 
\Big( \sum ^{N_1}_{i = 1}\hat{L}^{n, M, i}_{[nt]}\Big) 
\max _{0\leq k\leq [nt] - 1}|\hat{\psi }^{n, M}_k| + 
\sum ^{[nt] - 1}_{k = 0}|\pi ^{n, M}_k|\Big\} ^4\Big] \\\nonumber 
&\leq & 
8\E \Big [ 
\Big( \sum ^{N_1}_{i = 1}\hat{L}^{n, M, i}_{[nt]}\Big)^4 
\max _{0\leq k\leq [nt] - 1}|\hat{\psi }^{n, M}_k|^4\Big] + 
8N^3\sum ^{[nt] - 1}_{k = 0}\E [|\pi ^{n, M}_k|^4], 
\end{eqnarray*}
where 
\begin{eqnarray*}
\hat{\psi }^{n, M}_k &=& 
\sum ^{N_1}_{i = 1}\hat{\eta }^{n, M, i}_k + 
2\psi _M(X^{n, M}_{k/n})\sum ^N_{j = 1}\left (\frac{1}{\sqrt{n}}\tilde{g}^{n, i}_k(X^{n, M}) + 
\rho ^{n, M, i}_k\right ), \\
\pi ^{n, M}_k &=& 
\psi _M(X^{n, M}_{k/n})^2\left\{ \frac{2}{\sqrt{n}}
\sum ^N_{i, j = 1}\tilde{g}^i_k(X^{n, M})\rho ^{n, M, j}_k + 
\left (\sum ^N_{i = 1}\rho ^{n, M, i}_k\right )^2\right\}. 
\end{eqnarray*}
Since [A2] and Proposition \ref {prop_moment_diff_X} imply 
$\E [\psi _M(X^{n, M})(\rho ^{n, M, i}_k)^8] \leq C'_M/n^4$ 
and thus $\E [(\hat{\psi }^{n, M}_k)^8] \leq C''_M/n^4$, \ $\E [(\pi ^{n, M}_k)^4]\leq C''_M/n^2$ 
for some $C'_M, C''_M > 0$ (by virtue of Proposition \ref {prop_moment_eta}), 
using Proposition \ref {prop_bdd_L}, we get 
\begin{eqnarray*}
\E \Big [\Big (\sum ^{[nt] - 1}_{k = 0}|\hat{\varepsilon }^{n, M}_k|\Big) ^4\Big] &\leq & 
C'''_{M, t}\left \{ \E \Big [ 
\Big( \sum ^{N_1}_{i = 1}\hat{L}^{n, M, i}_{[nt]}\Big)^8 \Big]^{1/4}
\left( \sum ^{[nt]-1}_{k = 0}\E  [|\hat{\psi }^{n, M}_k|^8]\right) ^{1/4} + 
\frac{1}{n^2} \right\} \\
&\leq & C''''_{M, t}\left( \frac{1}{\sqrt{n}} + \frac{1}{n}\right) 
\end{eqnarray*}
for some $C'''_{M, t}, C''''_{M, t} > 0$. This implies the assertion. 
\end{proof}

\begin{prop} \ \label{prop_bdd_Z}
$\E [\sup _{0\leq t\leq T}|Z^{n, M}_t|^4] \leq C_{M, T}$ for some $C_{M, T} > 0$. 
\end{prop}

\begin{proof} 
Our assertion is obtained from [A2], Proposition \ref {prop_moment_eps}, (\ref {temp_diff_L_3}),
by calculating 
\begin{eqnarray*}
&&\E [\sup _{0\leq t\leq T}|Z^{n, M}_t|^4]\\&\leq & 
\E \left [\max _{0\leq k\leq [nT]}\left| \sum ^k_{l = 0}\left\{ \psi _M(X^{n, M}_{k/n})\left( 
\frac{1}{\sqrt{n}}\tilde{g}^{n, i}_k(X^{n, M}) + 
\frac{1}{n}h^{n, i}_k(X^{n, M})\right) + \tilde{\varepsilon }^{n, M, i}_k\right\} \right| ^4\right ]\\
&\leq & 
C'_{M, T}\left\{ \E [\max _{0\leq k\leq [nT]}|G^{n, M, i}_k|^4] + 
\frac{1}{n}\sum ^{[nT]}_{k = 0}\E [\sup _{|w|_\infty \leq M}|h^{n, M, i}_k|^4] + 
\E [ |\sum ^{[nT]}_{k = 0}\tilde{\varepsilon }^{n, M, i}_k| ^4] \right\} 
\end{eqnarray*}
for some $C'_{M, T} > 0$. 
\end{proof}

\begin{prop} \ \label{prop_tight}Let $(\xi ^n_k)_k$ be uniformly bounded random variables such that 
$\xi ^n_k$ is $\mathcal {G}^n_{k-1}$-adapted. 
(Here, $(\mathcal {G}^n_k)_k$ is defined as in the proof of Proposition \ref {prop_bdd_L}.) 
Put $\hat{Z}^{n, M, i}_t = \int ^t_0\xi ^n_{[nr]}dZ^{n, M, i}_r$. 
Then $(\hat{Z}^{n, M}_t)_t$ is tight on $\mathcal {C}$. 
\end{prop}

\begin{proof} 
It suffices to show 
\begin{eqnarray}\label{temp_tight_1_1}
\E [|\hat{Z}^{n, M, i}_u - \hat{Z}^{n, M, i}_t|^2|\hat{Z}^{n, M, i}_t - \hat{Z}^{n, M, i}_s|]\leq 
C_{M, T}(u-s)^{3/2}, \ \ n\in \Bbb {N} 
\end{eqnarray}
and 
\begin{eqnarray}\label{temp_tight_1_2}
\limsup _{n\rightarrow \infty }\E [|\hat{Z}^{n, M, i}_u - \hat{Z}^{n, M, i}_t|^2]\leq C_{M, T}(u-t)
\end{eqnarray}
for any $0 \leq s\leq t\leq u\leq T$ and for some $C_{M, T} > 0$ 
(cf. \cite {Kesten-Papanicolaou}). 
Set $\Phi  = |\hat{Z}^{n, M, i}_t - \hat{Z}^{n, M, i}_s|$ for brevity. 
Using the inequality 
\begin{eqnarray*}
\Big( \sum ^k_{l=1}x_l\Big) ^2 = \sum ^k_{l=1}x^2_l + 2\sum ^k_{l=1}x_l(x_1+\cdots + x_l),\ \ 
x_1,\ldots ,x_k\in \Bbb {R} 
\end{eqnarray*}
and the uniform boundedness of $(\xi ^n_k)_k$, we get 
\begin{eqnarray}\nonumber 
&&\E [|\hat{Z}^{n, M, i}_u - \hat{Z}^{n, M, i}_t|^2\Phi |]\\&\leq & 
3 \E \left [\left \{ |\xi ^n_{[nu]}H^{n, M, i}_{[nu]}|^2 + |\xi ^n_{[nt]}H^{n, M, i}_{[nt]}|^2 + 
\left( \sum ^{[nu]-1}_{k = [nt]}\xi ^n_kH^{n, M, i}_k\right)^2 \right\} \Phi \right ]\\\nonumber 
&\leq & 
C'\Bigg\{ \left (\E [(H^{n, M, i}_{[nu]})^4]^{1/2} + \E [(H^{n, M, i}_{[nt]})^4]^{1/2} + 
\sum ^{[nu] - 1}_{k = [nt]}\E [(H^{n, M, i}_k)^4]^{1/2}\right) \E [\Phi ^2]^{1/2}\\\label{temp_tight_Z_1}&&\hspace{12mm} + 
\sum ^{[nu] - 1}_{k = [nt]}\left |\E [\xi ^n_kH^{n, M, i}_k(X^{n, M})
(\hat{Z}^{n, M, i}_{k/n} - \hat{Z}^{n, M, i}_{[nt]/n})\Phi ]\right |\Bigg\} 
\end{eqnarray}
for some positive constant $C'$. 
By [A2] and Proposition \ref {prop_moment_eps}, 
we see that 
\begin{eqnarray}\label{temp_tight_Z_2}
\E [(H^{n, M, i}_k)^4] \leq \frac{C''_{M, T}}{n^2}, \ \ k \leq [nu] - 1 
\end{eqnarray}
for some $C''_{M, T} > 0$. 
Moreover, [A3] implies 
\begin{eqnarray*}
\E [\xi ^n_k\tilde{g}^{n, i}_k(X^{n, M})
(\hat{Z}^{n, M, i}_{k/n} - \hat{Z}^{n, M, i}_{[nt]/n})\Phi ] = 0, 
\end{eqnarray*}
and hence 
\begin{eqnarray}\nonumber 
&&\left |\E [\xi ^n_kH^{n, M, i}_k(X^{n, M})
(\hat{Z}^{n, M, i}_{k/n} - \hat{Z}^{n, M, i}_{[nt]/n})\Phi ]\right |\\\nonumber 
&\leq & 
\E \left [ \left| \frac{1}{n}\psi _M(X^{n, M}_{k/n})h^{n, i}_k(X^{n, M}) + 
\tilde{\varepsilon }^{n, M, i}_k\right|^4\right] ^{1/4}
\E [\sup _{0\leq k\leq [nu] - 1}
(\hat{Z}^{n, M, i}_{k/n} - \hat{Z}^{n, M, i}_{[nt]/n})^4]^{1/4}\E [\Phi ^2]^{1/2}\\\label{temp_tight_Z_3}
&\leq & 
C'''_{M, T}\left( \frac{1}{n} + \E [(\tilde{\varepsilon }^{n, M, i}_k)^4]^{1/4}\right) 
\E [\sup _{0\leq r\leq T}|Z^{n, M, i}_r|^4]^{1/4}\E [\Phi ^2]^{1/2} , \ \ k\leq [nu] - 1 
\end{eqnarray}
for some $C'''_{M, T} > 0$. 
The inequalities (\ref {temp_tight_Z_1})--(\ref {temp_tight_Z_3}) and 
Propositions \ref {prop_moment_eps}--\ref {prop_bdd_Z} imply 
\begin{eqnarray*}
\E [|\hat{Z}^{n, M, i}_u - \hat{Z}^{n, M, i}_t|^2\Phi ] \leq  
C''''_{M, T}\times \frac{[nu] - [nt] + 1}{n}\E [\Phi ^2]^{1/2} 
\end{eqnarray*}
for some $C''''_{M, T} > 0$. 
Replacing $\Phi $ with $1$ and performing the same calculation, we have 
\begin{eqnarray}\label{temp_tight_2_1}
\E [|\hat{Z}^{n, M, i}_t - \hat{Z}^{n, M, i}_s|^2]\leq C''''_{M, T}\times \frac{[nt] - [ns] + 1}{n}, 
\end{eqnarray}
hence 
\begin{eqnarray}\label{temp_tight_2_2}
\E [|\hat{Z}^{n, M, i}_u - \hat{Z}^{n, M, i}_t|^2|\hat{Z}^{n, M, i}_t - \hat{Z}^{n, M, i}_s|]\leq 
(C''''_{M, T})^{3/2}\left( \frac{[nu] - [ns] + 1}{n} \right) ^{3/2}. 
\end{eqnarray}
The inequality (\ref {temp_tight_2_1}) immediately leads to (\ref {temp_tight_1_2}). 
The inequality (\ref {temp_tight_1_1}) is now obtained by (\ref {temp_tight_2_2}) and 
the same argument as in the proof of Theorem 14.1 in \cite {Billingsley}. 
\end{proof}

As a consequence of Propositions \ref {prop_discrete_diff_l} and \ref {prop_tight}, 
we can see the tightness of the processes with fixed $M$. 

\begin{prop} \label {prop_tight2}
A family of processes $(X^{n, M}, Y^{n, M}, Z^{n, M}, L^{n, M}, \hat{L}^{n, M})_n$ is tight on 
$\mathcal {C}^{2(1 + N + N_1)}$. 
\end{prop}

\begin{proof} 
The tightness of $(Y^{n, M}, Z^{n, M})$ is obtained directly from Proposition \ref {prop_tight}. 
Then Theorem 7.3 in \cite {Billingsley} implies 
\begin{eqnarray}\label{temp_tight2_1_1}
\lim _{\delta \rightarrow 0}\limsup _{n\rightarrow \infty }
P(w_T(\delta  ; Y^{n, M, i}) \geq \varepsilon ') = 0, \ \ i = 0, \ldots , N 
\end{eqnarray}
for every $\varepsilon ' > 0$ and $T > 0$, where $w_T(\delta  ; x)$ stands for a modulus of continuity,
i.e., $w_T(\delta  ; x) = \sup _{0\leq s<t\leq T, |t-s| \leq  \delta }|x(t) - x(s)|$. 

For a while, we set $n$ sufficiently large so that $1/n < \delta $. 
Let $0\leq s < t\leq T$ be such that $|t-s|\leq \delta $. 
By [A1], we get 
\begin{eqnarray}\nonumber 
|\hat{L}^{n, M, i}_t - \hat{L}^{n, M, i}_s| &\leq & 
\hat{\eta }^{n, M, i}_{[nt]} + \hat{\eta }^{n, M, i}_{[ns]} + 
\sum ^{[nt]-1}_{k = [ns] + 1}
\left( 1 + \frac{\alpha ^{n, i}_k(X^{n, M})}{\bar{\alpha }^n_k(X^{n, M}) - \alpha ^{n, i}_k(X^{n, M})}\right) 
\eta ^{n, M, i}_k\\\label{temp_tight2_1_2}
&\leq & 
2\max _{0\leq k\leq [nT]}\hat{\eta }^{n, M, i}_k + \left( 1 + \frac{K_0}{(N-1)\delta _0}\right) 
w_T(\delta ; L^{n, M, i}). 
\end{eqnarray}
Similarly, Proposition \ref {prop_discrete_diff_l} implies 
\begin{eqnarray}\label{temp_tight2_1_3}
|L^{n, M, i}_t - L^{n, M, i}_s| \leq 
2\max _{0\leq k\leq [nT]}\eta ^{n, M, i}_k + 
\sqrt{\hat{K}}\sum ^{N_1}_{j = 1}w_T(\delta ; Y^{n, M, j}). 
\end{eqnarray}
By (\ref {temp_tight2_1_3}), Proposition \ref {prop_moment_eta}, and the Chebyshev inequality, it follows that 
\begin{eqnarray*}
P(w_T(\delta ; L^{n, M, i}) \geq \varepsilon ) &\leq & 
P\left (\sqrt{K_0}\sum ^{N_1}_{j = 1}w_T(\delta  ; Y^{n, M, j}) \geq \varepsilon /2\right ) + 
P(2\max _{0\leq k\leq [nT]}\eta ^{n, M, i}_k\geq \varepsilon /2)\\
&\leq & 
\sum ^{N_1}_{j = 1}P\left (w_T(\delta  ; Y^{n, M, j}) \geq \frac{\varepsilon }{2\sqrt{K_0}N_1}\right ) + 
\frac{256}{\varepsilon ^4}\sum ^{[nT]}_{k = 0}\E [(\eta ^{n, M, i}_k)^4]\\
&\leq & 
\sum ^{N_1}_{j = 1}P\left (w_T(\delta  ; Y^{n, M, j}) \geq \frac{\varepsilon }{2\sqrt{K_0}N_1}\right ) + 
\frac{256C_M([nT]+1)}{\varepsilon ^4n^2}
\end{eqnarray*}
for any $\varepsilon > 0$. 
Taking $\limsup _n$, letting $\delta \rightarrow 0$, and applying (\ref {temp_tight2_1_1}), we get 
\begin{eqnarray}\label{temp_tight2_2_1}
\lim _{\delta \rightarrow 0}\limsup _{n\rightarrow \infty }P(w_T(\delta  ; \hat{L}^{n, M, i}) \geq \varepsilon ) = 0. 
\end{eqnarray}
Similarly, (\ref {temp_tight2_1_3}), (\ref {temp_tight2_2_1}), and Proposition \ref {prop_moment_eta} imply 
\begin{eqnarray}\label{temp_tight2_2_2}
\lim _{\delta \rightarrow 0}\limsup _{n\rightarrow \infty }P(w_T(\delta  ; L^{n, M, i}) \geq \varepsilon ) = 0 
\end{eqnarray}
for any $\varepsilon > 0$. 
Furthermore, the inequality 
\begin{eqnarray*}
|X^{n, M}_t - X^{n, M}_s| \leq |Y^{n, M, 0}_t - Y^{n, M, 0}_s| + 
\frac{1}{(N-1)\delta _0}\sum ^{N_1}_{i = 1}|L^{n, M, i}_t - L^{n, M, i}_s|
\end{eqnarray*}
gives 
\begin{eqnarray}\label{temp_tight2_2_3}
\lim _{\delta \rightarrow 0}\limsup _{n\rightarrow \infty }P(w_T(\delta  ; X^{n, M}) \geq \varepsilon ) = 0. 
\end{eqnarray}
Our assertion is obtained from (\ref {temp_tight2_2_1})--(\ref {temp_tight2_2_3}) and the fact that 
the initial values $X^{n, M}_0$, $L^{n, M, i}_0, \allowbreak \hat{L}^{n, M, i}_0$ are all constants. 
\end{proof}

Proposition \ref {prop_tight2} tells us that 
for any sequence $(n_k)_k\subset \Bbb {N}$ there is a subsequence 
$(n_{k_l})_l\subset (n_k)_k$ such that 
$(X^{n_{k_l}, M}, Y^{n_{k_l}, M}, Z^{n_{k_l}, M}, L^{n_{k_l}, M}, \hat{L}^{n_{k_l}, M})$ 
converges weakly to a certain 
continuous process $(X^M, Y^M, Z^M, L^M, \hat{L}^M)$ defined on some probability space 
$(\Omega ^M, \mathcal {F}^M, P^M)$ 
on $\mathcal {C}^{2(1+ N + N_1)}$ as $l\rightarrow \infty $. 
Furthermore, since [A1] and [A4] imply that the convergences 
\begin{eqnarray*}
\tilde{\alpha }^{n, i}_{[nt]}(X^{n, M}) \longrightarrow \tilde {\alpha }^i(t, X^M), \ \ 
Q^{n, ij}_{[nt]}(X^{n, M}) \longrightarrow Q^{ij}(t, X^M), \ \ n\rightarrow \infty 
\end{eqnarray*}
are uniform in $t\in [0, T]$ for all $T > 0$, 
using Proposition \ref {prop_bdd_L}, 
Theorem 2.2 in \cite {Kurtz-Protter}, and the continuous mapping theorem 
(Theorem 2.7 in \cite {Billingsley}), 
we obtain the weak convergence of  
\begin{eqnarray}\nonumber 
&&\left( X^{n_{k_l}, M}, \varphi ^{n_{k_l}, M}, Y^{n_{k_l}, M}, 
Z^{n_{k_l}, M}, L^{n_{k_l}, M}, \tilde{I}^{n_{k_l}, M}, \tilde{J}^{n_{k_l}, M} \right) \\\label{convergence_weakly}
&&\longrightarrow 
\left(X^M, \varphi ^M, Y^M, Z^M, L^M, \tilde{I}^M, \tilde{J}^M\right), \ \ l\rightarrow \infty , 
\end{eqnarray}
where 
\begin{eqnarray*}
&&\tilde{I}^{n, M, i}_t = 
\int ^t_0\tilde {\alpha }^{n, i}_{[nr]}(X^{n, M})dL^{n, M, i}_r, \ \ 
\tilde{J}^{n, M, i}_t = 
\int ^t_0Q^{n, ij}_{[nt]}(X^{n, M})dL^{n, M, i}_r, \\
&&\tilde{I}^{M, i}_t = 
\int ^t_0\tilde {\alpha }^i(r, X^M)dL^{M, i}_r, \ \ 
\tilde{J}^{M, i}_t = 
\int ^t_0Q^{ij}(t, X^M)dL^{M, i}_r
\end{eqnarray*}
and 
\begin{eqnarray}\label{eq_phi_M}
\varphi ^{M, i}_t = Y^{M, i}_t + 1_{I_1}(i)L^{M,i}_t - \sum ^{N_1}_{j = 1}\tilde{J}^{M, j}_t. 
\end{eqnarray}
Note that (\ref {SP_like_X}) and (\ref {convergence_weakly}) tell us
\begin{eqnarray}\label{eq_X_M}
X^M_t = Y^{M, 0}_t + \tilde{I}^{M, i}_t. 
\end{eqnarray}
Let us introduce a filtration on $(\Omega ^M, \mathcal {F}^M, P^M)$. 
We define $\mathcal {G}^M_t = \sigma (X^M_r, Z^M_r, L^M_r ; r\leq t)$ and 
let $(\mathcal {F}^M_t)_t$ be an enlarged filtration of $(\mathcal {G}^M_t)_t$ such that 
$(\Omega ^M, \mathcal {F}^M, (\mathcal {F}^M_t)_t, P^M)$ satisfies the usual condition. 
We notice that the processes $\varphi ^M$, $Y^M$ and 
$\hat{L}^M$ are also $(\mathcal {F}^M_t)_t$-adapted. 
Now let us define 
\begin{eqnarray}\label{def_N}
N^{M, i}_t &=& Z^{M, i}_t - \int ^t_0\tilde{\beta }^{M, i}(r, X^M)dr, \\\nonumber 
\tilde{N}^{M, ij}_t &=& N^{M, i}_tN^{M, j}_t - \int ^t_0\psi _M(X^M_r)^2a^{ij}(r, X^M)dr, 
\end{eqnarray}
where $\tilde{\beta }^{M, i}(t, w) = \psi _M(w(t))\beta ^i(t, w) + \psi _M(w(t))^3\tilde{\gamma }^i(t, w)$. 

\begin{prop} \ \label{prop_martingale}
For each $i, j = 1, \ldots , N$, the processes $(N^{M, i}_t)_t$ and $(\tilde{N}^{M, ij}_t)_t$ are both $(\mathcal {F}^M_t)_t$-martingales. 
\end{prop}

\begin{proof} 
It suffices to show that 
$(N^{M, i}_t)_t$ and $(\tilde{N}^{M, ij}_t)_t$ are $(\mathcal {G}^M_t)_t$-martingales. 
Set 
\begin{eqnarray*}
&&N^{n, M, i}_t = \frac{1}{\sqrt{n}}\sum ^{[nt] - 1}_{k = 0}\tilde{g}^{n, i}_k(X^{n, M}) + 
\frac{nt-[nt]}{\sqrt{n}}\psi _M(X^{n, M}_{k/n})\tilde{g}^{n, i}_{[nt]}(X^{n, M}), \\
&&\bar{h}^{n, M, i}_k(w) = \E [h^{n, M, i}_k(w)], \ \ 
\tilde{\beta }^{n, M, i}_k(w) = \psi _M(w(k/n))\bar{h}^{n, M, i}_k(w). 
\end{eqnarray*}
Then we have 
\begin{eqnarray*}
&&\E \left[\sup _{0\leq t\leq T}\left| Z^{n, M, i}_t - N^{n, M, i}_t - 
\int ^t_0\tilde{\beta }^{n, M, i}_{[nr]}(X^{n, M})dr\right|^2 \right ]\\
&&\leq 
\frac{2}{n^2}\E \left[\max _{0\leq k\leq [nT]+1}
\left( \tilde{H}^{n, M, i}_k\right)^2 \right] + 
2\E [|\sum ^{[nt]}_{k = 0}\tilde{\varepsilon }^{n, M, i}_k|^2], 
\end{eqnarray*}
where 
\begin{eqnarray*}
\tilde{H}^{n, M, i}_k = \sum ^{k-1}_{l = 0}\psi _M(X^{n, M}_{k/n})(h^{n, M, i}_k(X^{n, M}) - 
\bar{h}^{n, M, i}_k(X^{n, M})). 
\end{eqnarray*}
Then the same calculation as (\ref {temp_diff_L_3}) and Proposition \ref {prop_moment_eps} leads us to 
\begin{eqnarray}\label{temp_mart_1_1}
\E \left[\sup _{0\leq t\leq T}\left| Z^{n, M, i}_t - N^{n, M, i}_t - 
\int ^t_0\tilde{\beta }^{n, M, i}_{[nr]}(X^{n, M})dr\right|^2 \right ] \ \longrightarrow \ 0, \ \ 
n\rightarrow \infty . 
\end{eqnarray}
Moreover [A4] implies 
$\lim _{n\rightarrow \infty }\tilde{\beta }^{n, M, i}_{[nt]}(w) = \tilde{\beta }^{M, i}(t, w)$ 
uniformly on any compact subset of $\mathcal {C}$ for all $t\geq 0$. 
Thus, using (\ref {convergence_weakly}), we get 
\begin{eqnarray}\nonumber 
&&\left( X^{n_{k_l}, M}, Z^{n_{k_l}, M}, L^{n_{k_l}, M}, \left( \int ^\cdot _0\tilde{\beta }^{n, M, i}_{[nr]}(X^{n, M})dr\right)_i \right) \\\label{temp_mart_1_2}
&&\longrightarrow 
\left( X^M, Z^M, L^M, \left( \int ^\cdot _0\tilde{\beta }^{M, i}(r, X^M)dr \right) _i \right) , \ \ l\rightarrow \infty 
\end{eqnarray}
weakly on $\mathcal {C}$. 
The convergences (\ref {temp_mart_1_1})--(\ref {temp_mart_1_2}) 
and the continuous mapping theorem imply that 
\begin{eqnarray}\label{temp_mart_1_3}
(X^{n_{k_l}, M}, Z^{n_{k_l}, M}, L^{n_{k_l}, M}, N^{n_{k_l}, M}) \ \longrightarrow \ 
(X^M, Z^M, L^M, N^M), \ \ l\rightarrow \infty \ \ 
\mathrm {weakly\ on\ } \mathcal {C}^{2N}. 
\end{eqnarray}
On the other hand, by [A3]--[A4], we have 
\begin{eqnarray*}
\E [(N^{n, M, i}_t - N^{n, M, i}_s)\Phi ((X^{n, M}_{s_l})^m_{l = 1}, (Z^{n, M}_{s_l})^m_{l = 1}, (L^{n, M}_{s_l})^m_{l = 1})] = 0
\end{eqnarray*}
and 
\begin{eqnarray*}
&&\E [(N^{n, M, i}_tN^{n, M, j}_t - N^{n, M, i}_sN^{n, M, j}_s)
\Phi ((X^{n, M}_{s_l})^m_{l = 1}, (Z^{n, M}_{s_l})^m_{l = 1}, (L^{n, M}_{s_l})^m_{l = 1})]\\
&=& 
\E [(N^{n, M, i}_t - N^{n, M, i}_s)(N^{n, M, j}_t - N^{n, M, j}_s)\Phi ]\\
&=& 
\int ^t_s\E [\psi _M(X^{n, M}_r)^2a^{n, ij}_{[nr]}(X^{n, M})\Phi ]dr\\&& - 
\frac{(nt-[nt])([nt]+1-nt)}{n}\E [\psi _M(X^{n, M}_{[nt]/n})a^{n, ij}_{[nt]}(X^{n, M})\Phi ]\\&& - 
\frac{(ns-[ns])([ns]+1-ns)}{n}\E [\psi _M(X^{n, M}_{[ns]/n})a^{n, ij}_{[ns]}(X^{n, M})\Phi ]
\end{eqnarray*}
for any $m\in \Bbb {N}$, $0\leq s_1 \leq \cdots \leq s < t$ and any bounded continous function 
$\Phi : \Bbb {R}^{(1+N+N_1)m}\longrightarrow \Bbb {R}$. 
Thus, using [A2], (\ref {temp_mart_1_3}), and the dominated convergence theorem, we obtain 
\begin{eqnarray}\label{temp_mart_N_1}
&&\E \hspace{0mm}^M[(N^{M, i}_t - N^{M, i}_s)\Phi ((X^M_{s_l})^m_{l = 1}, (Z^M_{s_l})^m_{l = 1}, (L^M_{s_l})^m_{l = 1})] = 0, \\\nonumber 
&&\E \hspace{0mm}^M[(N^{M, i}_tN^{M, j}_t - N^{M, i}_sN^{M, j}_s)\Phi ((X^M_{s_l})^m_{l = 1}, (Z^M_{s_l})^m_{l = 1}, (L^M_{s_l})^m_{l = 1})]\\\label{temp_mart_N_2}
&& = 
\E \hspace{0mm}^M[\int ^t_s\psi _M(X^M_r)^2a^{ij}(r, X^M)dr\Phi ((X^M_{s_l})^m_{l = 1}, (Z^M_{s_l})^m_{l = 1}, (L^M_{s_l})^m_{l = 1})] 
\end{eqnarray}
(where $\E \hspace{0mm}^M$ stands for the expectation under $P^M$), 
which imply our assertion. 
\end{proof}

By Proposition \ref {prop_martingale} and the martingale representation theorem 
(Theorem 3.4.2 in \cite {Karatzas-Shreve}), 
we can construct an enlarged filtered space 
$(\hat{\Omega }^M, \hat{\mathcal {F}}^M, (\hat{\mathcal {F}}^M_t)_t, \hat{P}^M)$ of 
$(\Omega ^M, \mathcal {F}^M, (\mathcal {F}^M_t)_t, P^M)$ and 
find an $N$-dimensional $(\hat{\mathcal {F}}^M_t)_t$-Brownian motion $(B^M_t)_t$ such that 
\begin{eqnarray}\label{mart_rep}
N^{M, i}_t = \sum ^N_{j = 1}\int ^t_0\psi _M(X^M_r)\sigma ^{ij}(r, X^M)dB^{M, j}_r, 
\end{eqnarray}
where the stochastic processes on 
$(\Omega ^M, \mathcal {F}^M, (\mathcal {F}^M_t)_t, P^M)$ 
are regarded as defined on \allowbreak 
$(\hat{\Omega }^M, \hat{\mathcal {F}}^M, \allowbreak (\hat{\mathcal {F}}^M_t)_t, \hat{P}^M)$ canonically. 
Moreover the process $(Z^{M, i}_t)_t$ becomes an $(\hat{\mathcal {F}}^M_t)_t$-semimartingale 
and we can define the stochastic integral
\begin{eqnarray*}
\int ^\cdot _0\xi _tdZ^{M, i}_t = \int ^\cdot _0\xi _t\tilde{\beta }^{M, i}(t, X^M)dt + 
\sum ^N_{j = 1}\int ^\cdot _0\xi _t\psi _M(X^M_t)\sigma ^{ij}(t, X^M)dB^{M, j}_t
\end{eqnarray*}
for an $(\hat{\mathcal {F}}^M_t)_t$-progressively measurable process $(\xi _t)_t$ 
(under suitable moment conditions).

\begin{prop} \ \label{prop_equal_Y_Z}The following equalities hold. 
\begin{eqnarray*}
Y^{M, 0}_t &=& x_0 + \sum ^N_{j = 1}\int ^t_0\frac{1}{\bar{\alpha }(r, X^M)}dZ^{M, j}_r, \\
Y^{M, i}_t &=& \Phi ^i + Z^{M, i}_t - \int ^t_0\frac{\alpha ^i(r, X^M)}{\bar{\alpha }(r, X^M)}dZ^{M, j}_r, \ \ 
i = 1, \ldots , N
\end{eqnarray*}
\end{prop}

\begin{proof} 
By Proposition \ref {prop_moment_eps} and
\begin{eqnarray*}
\frac{1}{n}\sum ^N_{i = 1}\sum ^{[nt] - 1}_{k = 0}\E [|\psi _M(X^{n, M}_{k/n})\tilde{g}^{n, i}_k(X^{n, M})|^2 + 
|\psi _M(X^{n, M}_{k/n})h^{n, M, i}_k(X^{n, M})|] \leq C_{M, t}
\end{eqnarray*}
for some $C_{M, t} > 0$, we can apply Theorem 2.2 in \cite {Kurtz-Protter} to arrive at the weak convergence of
\begin{eqnarray}\nonumber 
&&\left( Y^{n_{k_l}, M}, Z^{n_{k_l}, M}, 
\left( \int ^\cdot _0\frac{1}{\bar{\alpha }^{n_{k_l}}_{[n_{k_l}r]}(X^{n_{k_l}, M})}dZ^{n_{k_l}, M, i}_r\right)_i, 
\left( \int ^\cdot _0\frac{\alpha ^{n_{k_l}, i}(X^{n_{k_l}, M})}{\bar{\alpha }^{n_{k_l}}_{[n_{k_l}r]}(X^{n_{k_l}, M})}dZ^{n_{k_l}, M, i}_r\right)_i 
\right) \\\label{temp_equal}
&&\longrightarrow 
\left( Y^M, Z^M, 
\left( \int ^\cdot _0\frac{1}{\bar{\alpha }(r, X^M)}dZ^{M, i}_r\right)_i, 
\left( \int ^\cdot _0\frac{\alpha ^i(r, X^M)}{\bar{\alpha }(r, X^M)}dZ^{M, i}_r\right)_i 
\right) , \ \ l\rightarrow \infty . 
\end{eqnarray}
Our assertion is now obtained using 
(\ref {def_Yn1})--(\ref {def_Yn2}), (\ref {temp_equal}), and the continuous mapping theorem. 
\end{proof}

By (\ref {convergence_weakly})--(\ref {def_N}), (\ref {mart_rep}),
and Proposition \ref {prop_martingale}--\ref {prop_equal_Y_Z}, 
we obtain the following proposition. 

\begin{prop} \ \label{prop_SDER_M}
The pair $(X^M, \varphi ^M)$ (and the regulator process $L^M$) is the solution of 
\begin{eqnarray}\nonumber 
dX^M_t &=& \hat{b}^{M, 0}(t, X^M)dt + \sum ^N_{j = 1}\hat{\sigma }^{M, 0j}(t, X^M)dB^{M, j}_t + 
\sum ^{N_1}_{j = 1}\tilde{\alpha }^j(t, X^M)dL^{M, j}_t , 
\ \ X^M_0 = x_0, \\\nonumber 
d\varphi ^{M, i}_t &=& \hat{b}^{M, i}(t, X^M)dt + \sum ^N_{j = 1}\hat{\sigma }^{M, ij}(t, X^M)dB^{M, j}_t + 
1_{I_1}(i)dL^{M, i}_t\\\label{SDER_M}&& - \sum ^{N_1}_{j = 1}Q^{ij}(t, X^M)dL^{M, j}_t, 
\ \ \varphi ^i_0 = \Phi ^i,\ \ i=1, \ldots , N, 
\end{eqnarray}
where $\hat{b}^M$ and $\hat{\sigma }^M$ are given by 
$(\ref {def_hat_b})$--$(\ref {def_hat_sigma})$ upon replacing $\tilde{\beta }^i$ 
and $\sigma ^{ij}$ by $\tilde{\beta }^{M, i}$ and $\sigma ^{M, ij}$. 
\end{prop}

\begin{proof} 
It is obvious that $(\varphi ^{M, i}_t)_t$ is non-negative, 
$(L^{M, j}_t)_t$ is non-decreasing, and $L^{M, j}_0 = 0$ for $i = 1, \ldots , N$ and 
$j = 1, \ldots , N_1$. 
The rest of the proof is to show 
\begin{eqnarray}\label{temp_comp_0}
\int ^\infty _0\varphi ^{M, i}_rdL^{M, i}_r = 0, \ \ i = 1, \ldots , N_1\ \ 
\mbox {almost surely.} 
\end{eqnarray}
By the definition of $L^{n, M}$, we have 
\begin{eqnarray}\label{temp_comp_1}
\int ^T_0\varphi ^{n, M, i}_{([nr]+1)/n}dL^{n, M, i}_r = 
\sum ^{[nT]-1}_{l = 0}\varphi ^{n, M, i}_{(l+1)/n}\eta ^{n, M, i}_l + 
(nT - [nT])\varphi ^{n, M, i}_{([nT]+1)/n}\eta ^{n, M, i}_{[nT]} = 0, \ \ T \geq 0. 
\end{eqnarray}
Propositions \ref {prop_moment_eta}--\ref {prop_bdd_L} imply 
\begin{eqnarray}\nonumber 
&&\E [\sup _{0\leq t\leq T}
| \int ^t_0(\varphi ^{n, M, i}_{([nr]+1)/n} - \varphi ^{n, M, i}_r)dL^{n, M, i}_r|]\\\nonumber 
&\leq & 
\E [(L^{n, M, i}_T)^2]^{1/2}\left( \sum ^N_{j = 1}\E [\max _{0\leq k\leq [nT]}(H^{n, M, j}_k)^2]^{1/2} + 
\sum ^{N_1}_{j = 1}\E [\max _{0\leq k\leq [nT]}(\hat{\eta }^{n, M, j}_k)^2]^{1/2} \right) \\\nonumber 
&\leq & 
\E [(L^{n, M, i}_T)^2]^{1/2}
\left\{  \sum ^N_{j = 1}\left( \sum ^{[nT]}_{k = 0}\E [(H^{n, M, j}_k)^4]\right) ^{1/4} + 
\sum ^{N_1}_{j = 1}\left( \sum ^{[nT]}_{k = 0}\E [(\hat{\eta }^{n, M, j}_k)^4]\right) ^{1/4} \right\} \\\label{temp_comp_2}
&\leq & 
\frac{C_{M, T}}{n^{1/4}}\ \longrightarrow \ 0, \ \ n\rightarrow \infty , \ \ T > 0 
\end{eqnarray}
for some $C_{M, T} > 0$. 
Using (\ref {temp_comp_1})--(\ref {temp_comp_2}) 
and Theorem 2.2 in \cite {Kurtz-Protter}, we obtain (\ref {temp_comp_0}). 
\end{proof}

Here, (\ref {eq_phi_M}) and (\ref {temp_comp_0}) imply that 
$((\varphi ^{M, i})^{N_1}_{i=1}, (L^{M, i})^{N_1}_{i=1})$ is a solution of 
the Skorokhod problem associated with $(Y^{M, i})^{N_1}_{i = 1}$ 
(for given $X^M$).
Then, applying the standard argument of the Skorokhod problem, we get 
\begin{eqnarray*}
L^{M, i}_t = \sup _{0\leq r\leq t}
\left( \sum ^{N_1}_{j = 1}\int ^t_0Q^{ij}(r, X^M)dL^{M, j}_r - Y^{M, i}_t \right)_+ , \ \ 
i = 1, \ldots , N_1. 
\end{eqnarray*}
Hence, similar to Proposition \ref {prop_discrete_diff_l}, 
the same arguments as in the proof of Theorem 2 of \cite {Shashiashvili} leads us to 
\begin{eqnarray}\label{diff_L_M}
\sum ^{N_1}_{i = 1}|L^{M, i}_t - L^{M, i}_s|^2 \leq \hat{K}\sup _{s\leq r\leq t}
\sum ^{N_1}_{i = 1}|Y^{M, i}_r - Y^{M, i}_s|^2, \ \ 0\leq s < t 
\end{eqnarray}
for some $\hat{K} > 0$ which depends only on $V$. 

\begin{prop} \ \label{prop_bdd_XM}$\sup _M\E \hspace{0mm}^M[\sup _{0\leq t\leq T}|X^M_t|^4] < \infty $ for all $T>0$. 
\end{prop}

\begin{proof} 
Take any $R > 0$ and set $\tau _R = \inf \{ t\geq 0\ ; \ |X^M_t|\geq R \}$ and 
$m^R_t = \E \hspace{0mm}^M[\sup _{0\leq r\leq \min \{ t, \tau _R\} }|X^M_r|^4]$. 
From (\ref {diff_L_M}), we see that 
\begin{eqnarray}\nonumber 
&&\sum ^{N_1}_{i = 1}\sup _{s\leq r\leq t}\left| \int ^r_s\tilde{\alpha }^i(u, X^M)dL^{M, i}_u \right| ^2 \leq  
\frac{1}{(N-1)^2\delta ^2_0}\sum ^{N_1}_{i = 1}|L^{M, i}_t - L^{M, i}_s|^2
\\\label{ineq_diff_LM}&&\leq 
\frac{\hat{K}}{(N-1)^2\delta ^2_0}\sum ^{N_1}_{i = 1}\sup _{s\leq r\leq t}
\left| \int ^r_s\hat{b}^{M, i}(v, X^M)dv + 
\sum ^N_{j = 1}\int ^r_01_{[s, \infty )}(v)\hat{\sigma }^{M, ij}(v, X^M)dB^{M, j}_v\right| ^2, \hspace{15mm}
\end{eqnarray}
hence [A6], Proposition \ref {prop_SDER_M}, 
the H\"older inequality, and the Burkholder--Davis--Gundy inequality imply 
\begin{eqnarray*}
m^R_t &\leq & 
C\Bigg\{ 1 + T^3\sum ^N_{i = 0}\int ^t_0
\E \hspace{0mm}^M[1_{\{\tau _R\geq r\}}\sup _{0\leq s\leq r}|\hat{b}^{M, i}(s, X^M)|^4]dr\\&&\hspace{7mm} + 
T\sum ^N_{i = 0}\sum ^N_{j = 1}\int ^t_0
\E \hspace{0mm}^M[1_{\{\tau _R\geq r\}}\sup _{0\leq s\leq r}|\hat{\sigma }^{M, ij}(s, X^M)|^4]dr
\Bigg\} \\
&\leq & 
C'_T\left\{ 1 + \int ^t_0m^R_rdr\right\} , \ \ t\leq T
\end{eqnarray*}
for some $C > 0$ and $C'_T > 0$. Then we apply the Gronwall inequality to get 
$m^R_T = 0$. 
Our assertion is now obtained by letting $R\rightarrow \infty $. 
\end{proof}

The inequality (\ref {diff_L_M}) and Proposition \ref {prop_bdd_XM} immediately 
give the following proposition. 

\begin{prop} \ \label{prop_bdd_LM}$\sup _M\E \hspace{0mm}^M[(L^M_T)^4] < \infty $ for all $T>0$. 
\end{prop}

\begin{prop} \ \label{prop_tight_M}The family of processes $(X^M, \varphi ^M, Y^M, Z^M, L^M)_M$ 
is tight on $\mathcal {C}^{2 + 2N + N_1}$. 
\end{prop}

\begin{proof} 
By (\ref {diff_L_M}), (\ref {ineq_diff_LM}), and 
a calculation similar to that in the proof of the above proposition, we have 
\begin{eqnarray}\nonumber 
&&\E \hspace{0mm}^M[|X^M_t - X^M_s|^4] + 
\sum ^N_{i = 1}E \hspace{0mm}^M[|\varphi ^{M, i}_t - \varphi ^{M, i}_s|^4] + 
\sum ^{N_1}_{i = 1}\E \hspace{0mm}^M[|L^{M, i}_t - L^{M, i}_s|^4]\\\nonumber &&
\leq 
C\sum ^N_{i = 0}\E \hspace{0mm}^M[\sup _{s\leq r\leq t}|Y^{M, i}_t - Y^{M, i}_s|^4]\leq 
C'_T(t-s)\int ^t_s(1 + \E \hspace{0mm}^M[\sup _{0\leq v\leq r}|X^M_v|^4])dr\\\label{temp_tight_M_1}&&
\leq 
C'_T(1 + \sup _M\E \hspace{0mm}^M[\sup _{0\leq r\leq T}|X^M_r|^4])(t-s)^2, \ \ 0\leq s\leq t\leq T 
\end{eqnarray}
for some $C, C'_T > 0$. 
Similarly,
\begin{eqnarray}\label{temp_tight_M_2}
\E \hspace{0mm}^M[|Z^M_t - Z^M_s|^4] \leq 
C''_T(1 + \sup _M\E \hspace{0mm}^M[\sup _{0\leq r\leq T}|X^M_r|^4])(t-s)^2
\end{eqnarray}
for some $C''_T > 0$. 
The inequalities (\ref {temp_tight_M_1})--(\ref {temp_tight_M_2}), 
Proposition \ref {prop_bdd_XM}, and Theorem 2.3 in \cite {Stroock-Varadhan} then give the assertion. 
\end{proof}

\begin{proof}[Proof of Theorem \ref {th_converge}.] 
By Proposition \ref {prop_tight_M}, 
we see that for any non-decreasing sequence $(M_k)_k$ 
there is a subsequence $(M_{k_l})_l\subset (M_k)_k$ and continuous processes 
$(X, \varphi , Y, Z, L)$ on a certain probability space $(\Omega , \mathcal {F}, P)$ such that 
\begin{eqnarray}\label{weak_conv_M}
(X^{M_{k_l}}, \varphi ^{M_{k_l}}, Y^{M_{k_l}}, Z^{M_{k_l}}, L^{M_{k_l}})\ \longrightarrow \ 
(X, \varphi , Y, Z, L), \ \ l\rightarrow \infty \ \ \mbox {weakly}. 
\end{eqnarray}
We define $(N^i_t)_t$ by 
\begin{eqnarray}\label{def_N1}
N^i_t = Z^i_t - \int ^t_0\tilde{\beta }^i(r, X)dr. 
\end{eqnarray}
As in Proposition \ref {prop_martingale}, we get the weak convergence 
$(X^{M_{k_l}}, Z^{M_{k_l}}, L^{M_{k_l}}, N^{M_{k_l}}) \longrightarrow (X, Z, L, N)$ 
by Proposition \ref {prop_tight_M}. 
Thus, using (\ref {temp_mart_N_1})--(\ref {temp_mart_N_2}) and the martingale representation theorem, 
we can find an $N$-dimensional $(\mathcal {F}_t)_t$-Brownian motion $(B_t)_t$ on 
a certain filtered space $(\hat{\Omega }, \hat{\mathcal {F}}, (\hat{\mathcal {F}}_t)_t, \hat{P} )$,
which contains the original probability space $(\Omega , \mathcal {F}, P)$, such that 
\begin{eqnarray}\label{rep_N}
N^i_t = \sum ^N_{j = 1}\int ^t_0\sigma ^{ij}(r, X)dB^j_r. 
\end{eqnarray}
for $i = 1, \ldots , N$. 
As in Proposition \ref {prop_equal_Y_Z}, we get 
\begin{eqnarray}\label{eq_Y_Z}
Y^0_t = x_0 + \sum ^N_{j = 1}\int ^t_0\frac{1}{\bar{\alpha }(r, X)}dZ^j, \ \ 
Y^i_t = \Phi ^i + Z^i_t - \int ^t_0\frac{\alpha ^i(r, X)}{\bar{\alpha }(r, X)}dZ^j, \ \ i = 1, \ldots , N. 
\end{eqnarray}
Moreover, by (\ref {temp_comp_0}), Proposition \ref {prop_bdd_LM}, and Theorem 2.2 of \cite{Kurtz-Protter}, 
we get 
\begin{eqnarray}\label{property_phi_L}
\int ^\infty _0\varphi ^i_tdL^i_t = 0, \ \ i = 1, \ldots , N_1. 
\end{eqnarray}
By (\ref {weak_conv_M})--(\ref {property_phi_L}) and Proposition \ref {prop_SDER_M}, 
we see that $(X, \varphi , L)$ is a solution of our SDER (\ref {SDE_reflection}). 
Since [A7] implies that the distribution of $(X, \varphi , L)$ is uniquely determined, 
we get the weak convergence 
$(X^M, \varphi ^M, L^M)\longrightarrow (X, \varphi , L)$ as $M\rightarrow \infty $. 
The proof of Theorem \ref {th_converge} can be completed
using the arguments in step (vi) of the proof of Theorem 3 in \cite {Kesten-Papanicolaou}. 
\end{proof}

\end{document}